\documentclass[a4paper,12pt]{extarticle}

\usepackage{latexsym}
\usepackage{amscd}
\usepackage{graphics}
\usepackage{amsmath}
\usepackage{amssymb}
\usepackage{bbold}
\usepackage{mathrsfs}
\usepackage{amsthm}
\usepackage{xcolor}
\usepackage{accents}
\usepackage{enumerate}
\usepackage{url}
\usepackage{tikz-cd}
\usetikzlibrary{decorations.pathreplacing}
\usepackage{marginnote}
\usepackage{hyperref}
\usepackage{multicol,tikz}
\usetikzlibrary{calc}
\usepackage{marvosym}

\usepackage{newpxtext} \usepackage[euler-digits]{eulervm}

\theoremstyle{plain}
\newtheorem{theorem}{\sc Theorem}[section]
\makeatletter
\newcommand{\settheoremtag}[1]{
  \let\oldthetheorem\thetheorem
  \renewcommand{\thetheorem}{#1}
  \g@addto@macro\endtheorem{
    \addtocounter{theorem}{0}
    \global\let\thetheorem\oldthetheorem}
  }
\newtheorem{prop}[theorem]{\sc Proposition}
\newtheorem{lem}[theorem]{\sc Lemma}
\newtheorem{cor}[theorem]{\sc Corollary}
\theoremstyle{definition}
\newtheorem{defn}[theorem]{\sc Definition}
\newtheorem{rem}[theorem]{\sc Remark}

\DeclareMathOperator{\R}{\mathbb{R}}
\newcommand{\abs}[1]{\left\lvert#1\right\rvert}
\newcommand{\norm}[1]{\left\lVert#1\right\rVert}
\newcommand{\op}[1]{\operatorname{#1}}
\renewcommand{\hat}[1]{\widehat{#1}}
\numberwithin{equation}{section}
\renewcommand{\emptyset}{\varnothing}

\title{Selective symplectic homology with applications to contact non-squeezing}
\author{Igor Uljarevi\'c}
\date{June 2, 2023}

\usepackage{biblatex}
\begin{document}


\maketitle

\begin{abstract}
We prove a contact non-squeezing phenomenon on homotopy spheres that are fillable by Liouville domains with large symplectic homology: there exists a smoothly embedded ball in such a sphere that cannot be made arbitrarily small by a contact isotopy. These homotopy spheres include examples that are diffeomorphic to standard spheres and whose contact structures are homotopic to standard contact structures. As the main tool, we construct a new version of symplectic homology, called \emph{selective symplectic homology}, that is associated to a Liouville domain and an open subset of its boundary. The selective symplectic homology is obtained as the direct limit of Floer homology groups for Hamiltonians whose slopes tend to $+\infty$ on the open subset but remain close to 0 and positive on the rest of the boundary.
\end{abstract}
\section{Introduction}

One of the driving questions in contact geometry is how much it differs from smooth topology. How far does it go beyond topology? Does it, for instance, remember not only the shape but also the size of an object? In the absence of a natural measure, the size in contact geometry can conveniently be addressed via contact (non-)squeezing. We say that a subset $\Omega_a$ of a contact manifold $\Sigma$ can be contactly squeezed into a subset $\Omega_b\subset \Sigma$ if, and only if, there exists a contact isotopy $\varphi_t:\Sigma\to\Sigma, \: t\in[0,1]$ such that $\varphi_0=\op{id}$ and such that $\overline{\varphi_1(\Omega_a)}\subset \Omega_b$.
The most basic examples of contact manifolds are pessimistic as far as contact geometry and size are concerned. Namely, every bounded subset of the standard $\R^{2n+1}$ (considered with the contact form $dz +\sum_{j=1}^n \left( x_jdy_j - y_j dx_j\right)$) can be contactly squeezed into an arbitrarily small ball. This is because the map
\[ \R^{2n+1}\to\R^{2n+1}\quad:\quad (x,y,z)\mapsto \left(k\cdot x, k\cdot y, k^2\cdot z\right) \]
is a contactomorphism for all $k\in\R^+$. Consequently, every subset of a contact manifold whose closure is contained in a contact Darboux chart can be contactly squeezed into any non-empty open subset. In other words, contact geometry does not remember the size on a small scale. Somewhat surprisingly, this is not true on a large scale in general. 
In the next theorem, $B(R)$ denotes the ball of radius $R$.

\begin{theorem}[Eliashberg-Kim-Polterovich, Chiu]\label{thm:EKP}
The subset $\hat{B}(R) := B(R)\times\mathbb{S}^1$ of $\mathbb{C}^n\times \mathbb{S}^1$ can be contactly squeezed into itself via a compactly supported contact isotopy if, and only if, $R<1$.
\end{theorem}
This remarkable phenomenon, that may be seen as a manifestation of the Heisenberg uncertainty principle, was first observed by Eliashberg, Kim, and Polterovich \cite{eliashberg2006geometry}. They proved the case where either $R<1$ or $R\in\mathbb{N}.$ Chiu \cite{chiu2017nonsqueezing} extended their result to radii that are not necessarily integer. Fraser \cite{fraser2016contact} presented an alternative proof of the case of non-integer radii that is more in line with the techniques used in \cite{eliashberg2006geometry}. (Fraser actually proved the following formally stronger statement: there does not exist a compactly supported contactomorphism of $\mathbb{C}^n\times\mathbb{S}^1$ that maps the closure of $\hat{B}(R)$ into $\hat{B}(R)$ if $R\geqslant 1.$ It seems not to be known whether the group of compactly supported contactomorphisms of $\mathbb{C}^n\times\mathbb{S}^1$ is connected.) Using generating functions, Sandon reproved the case of integer radii \cite{sandon2011contact}.

The contact non-squeezing results are rare. Apart from Theorem~\ref{thm:EKP}, there are only few results about contact non-squeezing \cite{eliashberg2006geometry,albers2018orderability,allais2021contact,de2019orderability}, and they are all concerning the subsets of the form $ U\times\mathbb{S}^1$ in the prequantization of a Liouville manifold. The present paper provides examples of contact manifolds that are diffeomorphic to standard spheres and that exhibit non-trivial contact non-squeezing phenomena. The following theorem is the first example of contact non-squeezing for a contractible subset, namely an embedded standard smooth ball.

\begin{theorem}\label{thm:Ustilovskyspheres}
Let $\Sigma$ be an Ustilovsky sphere. Then, there exist two embedded closed balls $B_1, B_2\subset \Sigma$ of dimension equal to $\dim \Sigma$ such that $B_1$ cannot be contactly squeezed into $B_2$. 
\end{theorem}

An Ustilovky sphere is the $(4m+1)$-dimensional Brieskorn manifold
\[ \left\{ z=(z_0,\ldots, z_{2m+1})\in\mathbb{C}^{2m+2}\:|\: z_0^p + z_1^2 +\cdots + z_{2m+1}^2=0\:\&\: \abs{z}=1 \right\}\]
associated with natural numbers $m, p\in\mathbb{N}$ with $p\equiv \pm 1 \pmod{8} $. The Ustilovsky sphere is endowed with the contact structure given by the contact from
\[\alpha_p:= \frac{i p}{8}\cdot \left( z_0d\overline{z}_0-\overline{z}_0dz_0 \right) + \frac{i}{4}\cdot \sum_{j=1}^{2m+1}\left( z_jd\overline{z}_j-\overline{z}_jdz_j \right).\]
These Brieskorn manifolds were used by Ustilovsky \cite{ustilovsky1999infinitely} to prove the existence of infinitely many exotic contact structures on the standard sphere that have the same homotopy type as the standard contact structure. The strength of Theorem~\ref{thm:Ustilovskyspheres} lies in the topological simplicity of the objects used. A closed ball embedded in a smooth manifold can always be smoothly squeezed into an arbitrarily small (non-empty) open subset. Moreover, the obstruction to contact squeezing in Theorem~\ref{thm:Ustilovskyspheres} does not lie in the homotopy properties of the contact distribution. Namely, the contact distribution of an Ustilovsky sphere for $p\equiv 1 \pmod{2(2m)!}$ is homotopic to the standard contact distribution on the sphere and the contact non-squeezing on the standard contact sphere is trivial. A consequence of Theorem~\ref{thm:Ustilovskyspheres} is a contact non-squeezing on $\R^{4m+1}$ endowed with a non-standard contact structure.

\begin{cor}\label{cor:nonsqR}
Let $m\in\mathbb{N}$. Then, there exist a contact structure $\xi$ on $\R^{4m+1}$ and an embedded $(4m+1)$-dimensional closed ball $B\subset \R^{4m+1}$ such that $B$ cannot be squeezed into an arbitrary open non-empty subset by a compactly supported contact isotopy of $\left(\R^{4m+1}, \xi\right)$. 
\end{cor}

The exotic $\R^{4m+1}$ in Corollary~\ref{cor:nonsqR} is obtained by removing a point from an Ustilovsky sphere. In fact, the contact non-squeezing implies that $(\R^{4m+1}, \xi)$ constructed in this way (although tight) is not contactomorphic to the standard $\R^{4m+1}$. A more general result was proven by Fauteux-Chapleau and Helfer \cite{fauteux2021exotic} using a variant of contact homology: there exist infinitely many pairwise non-contactomorphic tight contact structures on $\R^{2n+1}$ if $n>1$. 

Theorem~\ref{thm:Ustilovskyspheres} is a consequence of the following theorem about homotopy spheres that bound Liouville domains with large symplectic homology.

\begin{theorem}\label{thm:homologyspheres}
Let $n> 2$ be a natural number and let $W$ be a $2n$-dimensional Liouville domain such that $\dim SH_\ast(W) > \sum_{j=1}^{2n} \dim H_j(W;\mathbb{Z}_2)$ and such that $\partial W$ is a homotopy sphere. Then, there exist two embedded closed balls $B_1, B_2\subset \partial W$ of dimension $2n-1$ such that $B_1$ cannot be contactly squeezed into $B_2$.
\end{theorem}

The smooth non-squeezing problem for a homotopy sphere is trivial: every non-dense subset of a homotopy sphere can be smoothly squeezed into an arbitrary non-empty open subset. This is due to the existence of Morse functions with precisely two critical points on the homotopy spheres. A smooth squeezing can be realized by the gradient flow of such a Morse function. Plenty of examples of Liouville domains that satisfy the conditions of Theorem~\ref{thm:homologyspheres} can be found among Brieskorn varieties. The Brieskorn variety $V(a_0,\ldots, a_m)$ is a Stein domain whose boundary is contactomorphic to the Brieskorn manifold $\Sigma(a_0,\ldots, a_m)$. Brieskorn \cite[Satz~1]{brieskorn1966beispiele} proved a simple sufficient and necessary condition (conjectured by Milnor) for a Brieskorn manifold to be homeomorphic to a sphere (see also \cite[Proposition~3.6]{kwon2016brieskorn}). Many of the corresponding Brieskorn varieties have infinite dimensional symplectic homology, for instance $V(3,2,2,\ldots,2)$. Thus, Theorem~\ref{thm:homologyspheres} also implies that there exists a non-trivial contact non-squeezing on the Kervaire spheres, i.e. on $\Sigma(3,2,\ldots, 2)$ for an odd number of 2's.

Our non-squeezing results are obtained using a novel version of symplectic homology, called \emph{selective symplectic homology}, that is introduced in the present paper. It resembles the relative symplectic cohomology by Varolgunes \cite{varolgunes2021mayer}, although the relative symplectic (co)homology and the selective symplectic homology are not quite the same. The selective symplectic homology, $SH_\ast^\Omega(W)$, is associated to a Liouville domain $W$ and an open subset $\Omega\subset \partial W$ of its boundary. Informally, $SH_\ast^{\Omega}(W)$ is defined as the Floer homology for a Hamiltonian on $W$ that is equal to $+\infty$ on $\Omega$ and to 0 on $\partial W\setminus \Omega$ (whereas, in this simplified view, the symplectic homology corresponds to a Hamiltonian that is equal to $+\infty$ everywhere on $\partial W$). The precise definition of the selective symplectic homology is given in Section~\ref{sec:SSH} below.

\sloppy The selective symplectic homology is related to the symplectic (co)homology of a Liouville sector that was introduced in \cite{ganatra2020covariantly} by Ganatra, Pardon, and Shende. As described in detail in \cite{ganatra2020covariantly}, every Liouville sector can be obtained from a Liouville manifold $X$ by removing the image of a stop. The notion of a stop on a Liouville manifold $X$ was defined by Sylvan \cite{sylvan2019partially} as a proper, codimension-0 embedding $\sigma: F\times\mathbb{C}_{\op{Re}<0}\to X$, where $F$ is a Liouville manifold, such that $\sigma^\ast \lambda_X= \lambda_F + \lambda_{\mathbb{C}} + df$, for a compactly supported $f$. Here, $ \lambda_X, \lambda_F, \lambda_{\mathbb{C}}$ are the Liouville forms on $X$, $F$, and $\mathbb{C}_{\op{Re}<0}$, respectively. We now compare the selective symplectic homology $SH_\ast^\Omega(W)$ and the symplectic homology $SH_\ast(X, \partial X)$, where $X= \hat{W}\setminus\op{im}\sigma$ is the Liouville sector obtained by removing a stop $\sigma$ from the completion $\hat{W}$, and $\Omega$ is the interior of the set $\partial W \setminus \op{im} \sigma$. Both $SH_\ast^\Omega(W)$ and $SH_\ast(X, \partial X)$  are, informaly speaking, Floer homologies for a Hamiltonian whose slope tends to infinity over $\Omega$. However, as opposed to $SH_\ast(X,\partial X)$, the selective symplectic homology $SH_\ast^\Omega(W)$ takes into account $\op{im} \sigma \cap W$, i.e. the part of the stop that lies outside of the conical end $\partial W\times(1,+\infty)$. Additionally, in the selective symplectic homology theory, there are no restrictions on $\Omega$: it can be any open subset, not necessarily the one obtained by removing a stop. On the technical side, $SH_\ast(X,\partial X)$ and $SH_\ast^\Omega(W)$ differ in the way the compactness issue is resolved. The symplectic homology of a Liouville sector is based on compactness arguments by Groman \cite{groman2015floer}, whereas the selective symplectic homology relies on a version of the Alexandrov maximum principle \cite[Theorem~9.1]{gilbarg1977elliptic}, \cite[Appendix~A]{abbondandolo2009estimates}, \cite{merry2019maximum}. It is an interesting question under what conditions $SH_\ast^\Omega(W)$ and $SH_\ast(X, \partial X)$ actually coincide. 

In simple terms, the non-squeezing results of the present paper are obtained by proving that a set $\Omega_b\subset \partial W$ with big selective symplectic homology cannot be contactly squeezed into a subset $\Omega_a\subset \partial W$ with $SH_\ast^{\Omega_a}(W)$ small (see Theorem~\ref{thm:ranknonsqueezing} on page~\pageref{thm:ranknonsqueezing}). The computation of the selective symplectic homology is somewhat challenging even in the simplest non-trivial cases. The key computations in the paper are that of $SH_\ast^D(W)$ where $D\subset\partial W$ is a contact Darboux chart, and that of $SH^{\partial W\setminus D}_\ast(W)$.  We prove that $SH_\ast^D( W)$ is isomorphic to $SH_\ast^\emptyset(W)$ by analysing the dynamics of a specific suitably chosen family of contact Hamiltonians that are supported in $D$ (see Theorem~\ref{thm:sshdarboux} on page~\pageref{thm:sshdarboux}). 
On the other hand, by utilizing the existence of a contractible loop of contactomorphisms that is positive over $D$, one can prove that $SH^{\partial W\setminus D}_\ast (W)$ is big if $SH_\ast(W)$ is big itself (see Section~\ref{sec:immaterial}). The proof is indirect and not quite straightforward. 
This proof also requires a feature of Floer homology for contact Hamiltonians that could be of interest in its own right and that has not appeared in the literature so far. Namely, there exists a collection of isomorphisms $\mathcal{B}(\sigma): HF_\ast(h)\to HF_\ast(h\# f)$ (one isomorphism for each admissible $h$) furnished by a family $\sigma$ of contactomorphisms of $\partial W$ that is indexed by a disc. In the formula above, $f$ is the contact Hamiltonian that generates the ``boundary loop'' of $\sigma$, and $h\#f$ is the contact Hamiltonian of the contact isotopy $\varphi^h_t\circ\varphi^f_t$. In addition, the isomorphisms $ \mathcal{B}(\sigma)$ give rise to an automorphism of the symplectic homology $SH_\ast(W)$.

\begin{rem}
    For the sake of simplicity, this paper defines the selective symplectic homology $SH_\ast^\Omega(W)$ in the framework of Liouville domains. The theory can actually be developed whenever $W$ is a symplectic manifold with contact type boundary such that the symplectic homology $SH_\ast(W)$ is well defined. This is the case, for instance, if $W$ is weakly+ monotone \cite{hofer1995floer} symplectic manifold with convex end. Theorem~\ref{thm:homologyspheres} and Theorem~\ref{thm:ranknonsqueezing} on page~\pageref{thm:ranknonsqueezing} are valid in this more general setting.
\end{rem}

What follows is a brief description of the main properties of the selective symplectic homology.

\subsection{Empty set}

The selective symplectic homology of the empty set is isomorphic, up to a shift in grading, to the singular homology of the Liouville domain relative its boundary:
\[ SH_\ast^{\emptyset}(W)\cong H_{\ast+ n} (W,\partial W; \mathbb{Z}_2),\]
where $2n=\dim W$. This is a straightforward consequence of the formal definition of the selective symplectic homology (Definition~\ref{def:SSH} on page \pageref{def:SSH}). Namely, it follows directly that $SH_\ast^\emptyset(W)$ is isomorphic to the Floer homology $HF_\ast(H)$ for a Hamiltonian $H_t:\hat{W}\to\R$ whose slope $\varepsilon>0$ is sufficiently small (smaller than any positive period of a closed Reeb orbit on $\partial W$). For such a Hamiltonian $H$, it is known (by a standard argument involving isomorphism of the Floer and Morse homologies for a $C^2$ small Morse function) that $HF_\ast(H)$ recoveres $H_{\ast+n}(W,\partial W;\mathbb{Z}_2)$.

\subsection{Canonical identification}\label{sec:canid}

Although not reflected in the notation, the group $SH_\ast^{\Omega}(W)$ depends only on the completion $\hat{W}$ and an open subset of the \emph{ideal contact boundary} of $\hat{W}$  (defined in \cite[page~1643]{eliashberg2006geometry}). More precisely,
$ SH_\ast^{\Omega}(W)= SH^{\Omega_f}_\ast(W^f),$
whenever  the pairs $(W, \Omega)$ and $(W^f, \Omega_f)$ are $\lambda$-related in the sense of the following definition.

\begin{defn}\label{def:lambdarel}
Let $(M,\lambda)$ be a Liouville manifold. Let $\Sigma_1,\Sigma_2\subset M$ be two hypersurfaces in $M$ that are transverse to the Liouville vector field. The subsets $\Omega_1\subset \Sigma_1$ and $\Omega_2\subset \Sigma_2$ are said to be $\lambda$-related if each trajectory of the Liouville vector field either intersects both $\Omega_1$ and $\Omega_2$ or neither of them.
\end{defn}

\subsection{Continuation maps}

To a pair $\Omega_a\subset \Omega_b$ of open subsets of $\partial W$, one can associate a morphism
\[\Phi=\Phi_{\Omega_a}^{\Omega_b} : SH_\ast^{\Omega_a}(W)\to SH_\ast^{\Omega_b}(W),\]
called \emph{continuation map}. The groups $SH_\ast^\Omega(W)$ together with the continuation maps form a directed system of groups indexed by open subsets of $\partial W$. In other words, $\Phi_{\Omega}^\Omega$ is equal to the identity and $\Phi_{\Omega_b}^{\Omega_c}\circ \Phi_{\Omega_a}^{\Omega_b}=\Phi_{\Omega_a}^{\Omega_c}$.

\subsection{Behaviour under direct limits}

Let $\Omega_k\subset \partial W$, $k\in\mathbb{N}$ be an increasing sequence of open subsets, i.e. $\Omega_k\subset \Omega_{k+1}$ for all $k\in\mathbb{N}$. Denote $\Omega:=\bigcup_{k=1}^{\infty} \Omega_k$. Then, the map
\[
\underset{k}{\lim_{\longrightarrow}}\: SH_\ast^{\Omega_k}(W) \to SH_\ast^{\Omega}(W),
\]
furnished by continuation maps is an isomorphism. The direct limit is taken with respect to continuation maps.

\subsection{Conjugation isomorphisms}\label{sec:conjugationiso}

The conjugation isomorphism
\[\mathcal{C}(\psi) : SH_\ast^{\Omega_a}(W)\to SH_\ast^{\Omega_b}(W)\]
is associated with a symplectomorphism $\psi:\hat{W}\to\hat{W}$, defined on the completion of $W$, that preserves the Liouville form outside of a compact set. With any such symplectomorphism $\psi$, one can associate a unique contactomorphism $\varphi:\partial W\to\partial W$, called \emph{ideal restriction}, such that
\[\psi(x,r)= \left( \varphi(x), f(x)\cdot r \right)\]
for $r\in\R^+$ large enough and for a certain positive function $f:\partial W\to \R^+$. The set $\Omega_b$ is the image of $\Omega_a$ under the contactomorphism $\varphi^{-1}:\partial W\to\partial W$. I.e. the conjugation isomorphism has the following form
\[\mathcal{C}(\psi) : SH_\ast^{\Omega}(W)\to SH_\ast^{\varphi^{-1}(\Omega)}(W),\]
where $\varphi$ is the ideal restriction of $\psi$. As a consequence, the groups $SH^{\Omega}_\ast(W)$ and $SH^{\varphi(\Omega)}_\ast(W)$ are isomorphic whenever the contactomorphism $\varphi$ is the ideal restriction of some symplectomorphism $\psi:\hat{W}\to\hat{W}$ (that preserves the Liouville form outside of a compact set). If a contactomorphism of $\partial W$ is contact isotopic to the identity, then it is equal to the ideal restriction of some symplectomorphism of $\hat{W}$. Hence, if $\Omega_a, \Omega_b\subset \partial W$ are two contact isotopic open subsets (i.e. there exists a contact isotopy $\varphi_t: \partial W\to \partial W$ such that $\varphi_0=\op{id}$ and such that $\varphi_1(\Omega_a)=\Omega_b$), then the groups $SH_\ast^{\Omega_a}(W)$ and $SH_\ast^{\Omega_b}(W)$ are isomorphic. The conjugation isomorphisms behave well with respect to the continuation maps, as asserted by the next theorem.

\begin{theorem}\label{thm:conjVSsont}
Let $W$ be a Liouville domain, let $\psi:\hat{W}\to\hat{W}$ be a symplectomorphism that preserves the Liouville form outside of a compact set, and let $\varphi:\partial W\to\partial W$ be the ideal restriction of $\psi$. Let $\Omega_a\subset \Omega_b\subset \partial W$ be open subsets. Then, the following diagram, consisting of conjugation isomorphisms and continuation maps, commutes
\[\begin{tikzcd}
SH_\ast^{\Omega_a}(W) \arrow{r}{\mathcal{C}(\psi)}\arrow{d}{\Phi}& SH_\ast^{\varphi^{-1}(\Omega_a)}(W)\arrow{d}{\Phi}\\
SH_\ast^{\Omega_b}(W) \arrow{r}{\mathcal{C}(\psi)}& SH_\ast^{\varphi^{-1}(\Omega_b)}(W).
\end{tikzcd}\]
\end{theorem}

\subsection*{Applications}

The selective symplectic homology is envisioned as a tool for studying contact geometry and dynamics of Liouville fillable contact manifolds. The present paper shows how it can be used to prove contact non-squeezing type of results. This is illustrated by the following abstract observation.

\begin{theorem}\label{thm:ranknonsqueezing}
Let $W$ be a Liouville domain and let $\Omega_a, \Omega_b\subset \partial W$ be open subsets. If the rank of the continuation map
$SH_\ast^{\Omega_b}(W)\to SH_\ast(W)$
is (strictly) greater than the rank of the continuation map
$SH_\ast^{\Omega_a}(W)\to SH_\ast(W),$
then $\Omega_b$ cannot be contactly squeezed into $\Omega_a$.
\end{theorem}

The theory of selective symplectic homology has rich algebraic structure that is beyond the scope of the present paper. For instance,
\begin{enumerate}
    \item one can construct a persistent module associated to an open subset of a contact manifold,
    \item topological quantum field theory operations are well defined on $SH_\ast^\Omega(W),$
    \item it is possible to define transfer morphisms for selective symplectic homology in analogy to Viterbo's transfer morphisms for symplectic homology,
    \item there exist positive selective symplectic homology, $\mathbb{S}^1$-equivariant selective symplectic homology, positive  $\mathbb{S}^1$-equivariant selective symplectic homology...
\end{enumerate}

\subsection*{The structure of the paper}
The paper is organized as follows. Section~\ref{sec:prelim} recalls the definition of Liouville domains and construction of the Hamiltonian-loop Floer homology. Sections~\ref{sec:SSH} - \ref{sec:conjugationisomorphisms} define the selective symplectic homology and derive its properties. Sections~\ref{sec:darboux} - \ref{sec:main} contain proofs of the applications to the contact non-squeezing and necessary computations. Section~\ref{sec:pathiso} discusses isomorphisms of contact Floer homology induced by families of contactomorphisms indexed by a disc.

\subsection*{Acknowledgements}
I would like to thank Paul Biran and Leonid Polterovich for their interest in this work and for valuable suggestions. This research was supported by the Science Fund of the Republic of Serbia, grant no.~7749891, Graphical Languages - GWORDS.

\section{Preliminaries}\label{sec:prelim}

\subsection{Liouville manifolds}

This section recalls the notions of a Liouville domain and a Liouville manifold of finite type. Liouville manifolds (of finite type) play the role of an ambient space in this paper. The selective symplectic homology is built from objects on a Liouville manifold of finite type.

\begin{defn}
A Liouville manifold of finite type is an open manifold $M$ together with a 1-form $\lambda$ on it such that the following conditions hold.
\begin{enumerate}
    \item The 2-form $d\lambda$ is a symplectic form on $M.$
    \item \sloppy There exist a contact manifold $\Sigma$ with a contact form $\alpha$ and a codimension-0 embedding
    $ \iota : \Sigma\times\R^+\to M $
    such that $M\setminus \iota(\Sigma\times\R^+)$ is a compact set, and such that $\iota^\ast \lambda=r\cdot \alpha,$ where $r$ stands for the $\R^+$ coordinate.
\end{enumerate}
\end{defn}
We will refer to the map $\iota$ as a \emph{conical end} of the Liouville manifold $M.$ With slight abuse of terminology, the set $\iota(\Sigma\times \R^+)$ will also be called \emph{conical end}. A conical end is not unique.

The Liouville vector field, $X_\lambda,$ of the Liouville manifold $(M, \lambda)$ of finite type is the complete vector field defined by $d\lambda(X_\lambda, \cdot)=\lambda.$ If $\Sigma\subset M$ is a closed hypersurface that is transverse to the Liouville vector field $X_\lambda,$ then $\left.\lambda\right|_{\Sigma}$ is a contact form on $\Sigma$ and there exists a unique codimension-0 embedding
$ \iota_\Sigma: \Sigma\times\R^+\to M $
such that $\iota_\Sigma(x,1)=x$ and such that $\iota_\Sigma^\ast\lambda= r\cdot \left.\lambda\right|_{\Sigma}$.

The notion of a Liouville manifold of finite type is closely related to that of a Liouville domain.

\begin{defn}
A Liouville domain is a compact manifold $W$ (with boundary) together with a 1-form $\lambda$ such that
\begin{enumerate}
    \item $d\lambda$ is a symplectic form on $W,$
    \item the Liouville vector field $X_\lambda$ points transversely outwards at the boundary. 
\end{enumerate}
\end{defn}

The Liouville vector field on a Liouville domain $(W,\lambda)$ is not complete. The completion of the Liouville domain is the Liouville manifold $(\hat{W},\hat{\lambda})$ of finite type obtained by extending the integral curves of the vector field $X_\lambda$ towards $+\infty.$ Explicitly, as a topological space,
\[\hat{W}\quad:=\quad W\quad\cup_{\partial}\quad (\partial W)\times [1,+\infty).\]
The manifolds $(\partial W)\times [1,+\infty)$ and $W$ are glued along the boundary via the map
\[\partial W\times\{1\}\to\partial W\quad:\quad (x,1)\mapsto x.  \]

The completion $\hat{W}$ is endowed with the unique smooth structure such that the natural inclusions $W\hookrightarrow \hat{W}$ and $\partial W\times [1, +\infty)\hookrightarrow \hat{W}$ are smooth embeddings, and such that the vector field $X_\lambda$ extends smoothly to $\partial W\times [1,+\infty)$ by the vector field $r\partial_r.$ (Here, we tacitly identified $\partial W\times [1,+\infty)$ and $W$ with their images under the natural inclusions.) The 1-form $\hat{\lambda}$ is obtained by extending the 1-form $\lambda$ to $\partial W\times[1,+\infty)$ by $r\cdot \left.\lambda\right|_{\partial W.}$
The completion of a Liouville domain is a Liouville manifold of finite type. And, other way around, every Liouville manifold of finite type is the completion of some Liouville domain.

Let $M$ be a Liouville manifold of finite type, let $W\subset M$ be a codimension-0 Liouville subdomain, and let $f:\partial W\to\R^+$ be a smooth function. The completion $\hat{W}$ can be seen as a subset of $M$. Throughout the paper, $W^f$ denotes the subset of $M$ defined by
\[W^f:=\hat{W}\setminus\iota_{\partial W}\big(\{f(x)\cdot r>1\}\big).\]
Here, $\{f(x)\cdot r>1\}$ stands for $\left\{(x,r)\in\partial W\times \R^+\:|\: f(x)\cdot r>1\right\}$. The set $W^f$ is a codimension-0 Liouville subdomain in its own right, and the completions of $W$ and $W^f$ can be identified.

\subsection{Floer theory}

In this section, we recall the definition of the Floer homology for a contact Hamiltonian, $HF_\ast(W,h).$ A contact Hamiltonian is called admissible if it does not have any $1$-periodic orbits and if it is 1-periodic in the time variable. The group $HF_\ast(W,h)$ is associated to a Liouville domain $(W,\lambda)$ and to an admissible contact Hamiltonian $h_t:\partial W\to \R$ that is defined on the boundary of $W.$ 

The Floer homology for contact Hamiltonians was introduced in \cite{merry2019maximum} by Merry and the author. It relies heavily on the Hamiltonian loop Floer homology \cite{floer1989symplectic} and  symplectic homology \cite{floer1994symplectic,floer1994applications,cieliebak1995symplectic,cieliebak1996applications,viterbo1999functors,viterbo2018functors}, especially the version of symplectic homology by Viterbo \cite{viterbo1999functors}.

\subsubsection{Auxiliary data}

Let $(W,\lambda)$ be a Liouville domain, and let $h_t:\partial W\to \R$ be an admissible contact Hamiltonian. The group $HF_\ast(W, h)$ is defined as the Hamiltonian loop Floer homology, $HF_\ast(H,J),$ associated to a Hamiltonian $H$ and an almost complex structure $J.$ Both $H$ and $J$ are objects on the completion $\hat{W}=:M$ of the Liouville domain $W.$ Before stating the precise conditions that $H$ and $J$ are assumed to satisfy, we define the set $\mathcal{J}(\Sigma, \alpha)$ of almost complex structures of \emph{SFT type}. Let $\Sigma$ be a contact manifold with a contact form $\alpha$. The set $\mathcal{J}(\Sigma, \alpha)$ (or simply $\mathcal{J}(\Sigma)$ when it is clear from the context what the contact form is equal to) is the set of almost complex structures $J$ on the symplectization $\Sigma\times\R^+$ such that
\begin{itemize}
    \item $J$ is invariant under the $\R^+$ action on $\Sigma\times\R^+$,
    \item $J(r\partial_r)= R_\alpha$, where $R_\alpha$ is the Reeb vector field on $\Sigma$ with respect to the contact form $\alpha$,
    \item the contact distribution $\xi:=\ker \alpha $ is invariant under $J$ and $\left.J\right|_{\xi}$ is a compatible complex structure on the symplectic vector bundle $(\xi, d\alpha)\to \Sigma$.
\end{itemize}
The list of the conditions for $(H,J)$ follows.
\begin{enumerate}
    \item (Conditions on the conical end). There exist a positive number $a\in\R^+$ and a constant $c\in\R$ such that
    \[H_t\circ\iota_{\partial W}(x,r)= r\cdot h(x) + c,\]
    for all $t\in\R$ and $(x,r)\in\partial W\times[a,+\infty),$ and such that $\iota_{\partial W}^\ast J_t$ coincides with an element of $\mathcal{J}(\partial W)$ on $\partial W\times [a,+\infty)$ for all $t\in\R$. Here, $\iota_{\partial W}: \partial W\times\R^+\to M$ is the conical end of $M$ associated to $\partial W.$
    \item (One-periodicity). For all $t\in\R,$ $H_{t+1}=H_t$ and $J_{t+1}=J_t.$
    \item ($d\hat{\lambda}$-compatibility). $d\hat{\lambda}(\cdot, J_t\cdot)$ is a Riemannian metric on $M$ for all $t\in\R.$
\end{enumerate}

The pair $(H,J)$ that satisfies the conditions above is called \emph{Floer data} (for the contact Hamiltonian $h$ and the Liouville domain $(W,\lambda)$). Floer data $(H,J)$ is called \emph{regular} if, additionally, the following two conditions hold.

\begin{enumerate}
    \setcounter{enumi}{3}
    \item (Non-degeneracy). The linear map
    \[ d\phi^H_1(x)-\op{id}\quad:\quad T_xM\to T_xM \]
    is invertible for all fixed points $x$ of $\phi_1^H.$
    \item(Regularity). The linearized operator of the Floer equation
    \[ u:\R\times (\R/\mathbb{Z})\to M,\quad \partial_s u+ J_t(u)(\partial_t u- X_{H_t}(u))=0 \]
    is surjective.
\end{enumerate}

\subsubsection{Floer complex}

Let $(H,J)$ be regular Floer data. The Floer complex, $CF_\ast(H,J),$ is built up on the contractible 1-periodic orbits of the Hamiltonian $H$. For every 1-periodic orbit $\gamma$ of the Hamiltonian $H,$ there exists a fixed point $x$ of $\phi^H_1$ such that $\gamma(t)=\phi^H_t(x).$ The degree, $\deg\gamma=\deg_H\gamma,$ of a contractible 1-periodic orbit $\gamma=\phi^H_\cdot(x)$ of the Hamiltonian $H$ is defined to be the negative Conley-Zehnder index of the path of symplectic matrices that is obtained from $d\phi^H_t(x)$ by trivializing $TM$ along a disc that is bounded by $\gamma$ (see \cite{salamon1999lectures} for details concerning the Conley-Zehnder index). Different choices of the capping disc can lead to different values of the degree, however they all differ by an even multiple of the minimal Chern number
\[N:=\min \left\{ c_1(u)>0\:|\: u:\mathbb{S}^2\to M \right\}.\]
Therefore, $\deg \gamma$ is well defined as an element of $\mathbb{Z}_{2N}$ (but not as an element of $\mathbb{Z},$ in general). The Floer chain complex as a group is defined by
\[CF_k(H,J):=\bigoplus_{\deg \gamma=k} \mathbb{Z}_2\left\langle\gamma\right\rangle.\]
Since the Floer data $(H,J)$ is regular, the set $\mathcal{M}(H,J, \gamma^-, \gamma^+)$ of the solutions  $u:\R\times(\R/\mathbb{Z})\to M$ of the Floer equation
\[ \partial_s u + J_t(u)(\partial_t u - X_{H_t}(u))=0\]
that join two 1-periodic orbits $\gamma^-$ and $\gamma^+$ of $H$ (i.e. $\displaystyle \lim_{s\mapsto\pm\infty} u(s,t)=\gamma^\pm(t)$) is a finite dimensional manifold (components of which might have different dimensions). There is a natural $\R$-action on $\mathcal{M}(H,J, \gamma^-, \gamma^+)$ given by
\[ \R\:\times\: \mathcal{M}(H,J, \gamma^-, \gamma^+)\quad\mapsto\quad \mathcal{M}(H,J, \gamma^-, \gamma^+)\quad :\quad (a, u)\mapsto u(\cdot +a, \cdot). \]
The quotient
\[\tilde{\mathcal{M}}(H,J,\gamma^-,\gamma^+):=\mathcal{M}(H,J,\gamma^-,\gamma^+)/\mathbb{R}\]
of $\mathcal{M}(H,J,\gamma^-,\gamma^+)$ by this action is also a finite dimensional manifold. Denote by $n(\gamma^-, \gamma^+)=n(H,J, \gamma^-, \gamma^+)\in\mathbb{Z}_2$ the parity of the number of 0-dimensional components of $\tilde{\mathcal{M}}(H,J,\gamma^-,\gamma^+).$ The boundary map 
\[\partial : CF_{k+1}(H,J)\to CF_k(H,J)\]
is defined on the generators by
\begin{equation}\label{eq:boundary}\partial \left\langle \gamma\right\rangle:=\sum_{\tilde{\gamma}} n(\gamma,\tilde{\gamma})\left\langle \tilde{\gamma} \right\rangle.\end{equation}
\sloppy If $\deg\gamma\not=\deg\tilde{\gamma}+1$, there are no 0-dimensional components of $\tilde{\mathcal{M}}(H,J,\gamma^-,\gamma^+)$, and therefore, $n(\gamma,\tilde{\gamma})=0.$ Hence, the sum in \eqref{eq:boundary} can be taken only over $\tilde{\gamma}$ that satisfy $\op{deg}\tilde{\gamma}=\op{deg}\gamma-1$. The homology of the chain complex $CF_\ast(H,J)$ is denoted by $HF_\ast(H,J).$

\subsubsection{Continuation maps}

Continuation maps compare Floer homologies for different choices of Floer data. They are associated to generic monotone homotopies of Floer data that join two given instances of Floer data. We refer to these homotopies as continuation data.

Let $(H^-, J^-)$ and $(H^+, J^+)$ be regular Floer data. The continuation data from $(H^-, J^-)$ to $(H^+, J^+)$ is a pair $(\{H_{s,t}\}, \{J_{s,t}\})$ that consists of an $s$-dependent Hamiltonian $H_{s,t}:M\to\R$ and a family $J_{s,t}$ of almost complex structures on $M$ such that the following conditions hold:

\begin{enumerate}
    \item (Homotopy of Floer data). For all $s\in\R,$ the pair $(H_{s,\cdot}, J_{s,\cdot})$ is Floer data (not necessarily regular) for some contact Hamiltonian.
    \item (Monotonicity). There exists $a\in\R^+$ such that 
    $\partial_s H_{s,t}(x)\geqslant0,$ 
    for all $s,t\in\R$ and $x\in\iota_{\partial W}(\partial W\times [a,+\infty)).$
    \item ($s$-independence at the ends). There exists $b\in\R^+$ such that 
    $H_{s,t}(x)= H^{\pm}_t(x),$
    for all $t\in \R$ and $x\in M$, if $\pm s\in [b,+\infty)$. 
\end{enumerate}

Continuation data $(\{H_{s,t}\},\{J_{s,t}\})$ is called \emph{regular} if the linearized operator of the $s$-dependent Floer equation
\[ u:\R\times (\R/\mathbb{Z})\to M,\quad \partial_s u+ J_{s,t}(u)(\partial_t u- X_{H_{s,t}}(u))=0 \]
is surjective.

Given regular continuation data $(\{H_{s,t}\}, \{J_{s,t}\})$ from $(H^-, J^-)$ to $(H^+, J^+)$ and 1-periodic orbits $\gamma^-$ and $\gamma^+$ of $H^-$ and $H^+,$ respectively, the set of the solutions $u:\R\times(\R/\mathbb{Z})\to M$ of the problem
\begin{align*}
    & \partial_s u + J_{s,t} (u) (\partial_t u - X_{H_{s,t}}(u))=0,\\
    & \lim_{s\to\pm\infty} u(s,t)= \gamma^\pm(t)
\end{align*}
is a finite dimensional manifold. Its 0-dimensional part is compact, and therefore, a finite set. Denote by $m(\gamma^-,\gamma^+)$ the number modulo 2 of the 0-dimensional components of this manifold. The continuation map
\[\Phi= \Phi(\{H_{s,t}\}, \{J_{s,t}\})\quad:\quad CF_\ast(H^-, J^-)\to CF_\ast(H^+, J^+)\]
is the chain map defined on the generators by
\[\Phi(\gamma^-):=\sum_{\gamma^+} m(\gamma^-, \gamma^+)\left\langle \gamma^+\right\rangle.\]
The map $HF_\ast(H^-, J^-)\to HF_\ast(H^+, J^+)$ induced by a continuation map on the homology level (this map is also called \emph{continuation map}) does not depend on the choice of continuation data from $(H^-, J^-)$ to $(H^+, J^+).$

The groups $HF_\ast(H,J)$ together with the continuation maps form a directed system of groups. As a consequence, the groups $HF_\ast(H,J)$ and $HF_\ast(H', J')$ are canonically isomorphic whenever $(H,J)$ and $(H',J')$ are (regular) Floer data for the same admissible contact Hamiltonian. Therefore, the Floer homology $HF_\ast(h)= HF_\ast(W,h)$ for an admissible contact Hamiltonian $h_t:\partial W\to\R$ is well defined. The continuation maps carry over to Floer homology for contact Hamiltonians. Due to the ``monotonicity''condition for the continuation data, the continuation map 
$HF_\ast(h)\to HF_\ast(h')$
is not well defined unless $h_t,h'_t:\partial W\to\R$ are admissible contact Hamiltonians such that $h\leqslant h',$ pointwise.

For a positive smooth function $f:\partial W\to \R^+$, the completions of the Liouville domains $W$ and $W^f$ can be naturally identified. If a Hamiltonian $H: \hat{W}= \hat{W^f}\to \R$ has the slope equal to $h$ with respect to the Liouville domain $W^f$, then it has the slope equal to $f\cdot h$ with respect to the Liouville domain $W$. Therefore, the groups $HF_\ast(W^f, h)$ and $HF_\ast(W, f\cdot h)$ are canonically isomorphic. Here, we tacitly identified $\partial W$ and $\partial W^f$ via the contactomorphism furnished by the Liouville vector field, and regarded $h$ as both the function on $\partial W$ and $\partial W^f$.

\section{Selective symplectic homology}\label{sec:SSH}

This section defines formally the selective symplectic homology $SH_\ast^{\Omega}(W)$. To this end, two sets of smooth functions on $\partial W$ are introduced : $\mathcal{H}_\Omega(\partial W)$ and $\Pi(h)$. The set $\mathcal{H}_\Omega(\partial W)$ consists of certain non-negative smooth functions on $\partial W$, and $\Pi(h)$ is a set associated to $h\in \mathcal{H}_\Omega(\partial W)$ that can be thought of as the set of perturbations. 

\begin{defn}\label{def:Hasigma}
Let $\Sigma$ be a closed contact manifold with a contact form $\alpha,$ and let $\Omega\subset \Sigma$ be an open subset. Denote by $\mathcal{H}_\Omega(\Sigma)= \mathcal{H}_\Omega(\Sigma,\alpha)$ the set of smooth ($C^\infty$) autonomous contact Hamiltonians $h:\Sigma\to[0,+\infty)$ such that
\begin{enumerate}
    \item $ \op{supp} h\subset \Omega$,\label{cond:van}
    \item $dY^h(p)=0$ for all $p\in \Sigma$ such that $h(p)=0$,
    \item the 1-periodic orbits of $h$ are constant.
\end{enumerate}
\end{defn}
In the definition above, $Y^h$ denotes the contact vector field of the contact Hamiltonian $h$. More precisely, the vector field $Y^h$ is determined by the following relations
\begin{align*}
    & \alpha(Y^h)=-h,\\
    & d\alpha(Y^h, \cdot)= dh- dh(R)\cdot \alpha,
\end{align*}
where $R$ stands for the Reeb vector field with respect to $\alpha$. The condition $dY^h(p)=0$ holds for $p\in h^{-1}(0)$ if, for instance, the Hessian of $h$ is equal to 0 at the point $p$. The set $\mathcal{H}_\Omega(\Sigma)$ is non-empty.
\begin{defn}\label{def:Pih}
Let $\Sigma$ be a closed contact manifold with a contact form $\alpha,$ let $\Omega\subset \Sigma$ be an open subset, and let $h\in\mathcal{H}_\Omega(\Sigma).$ Denote by $\Pi(h)$ the set of smooth positive functions $f:\Sigma\to\R^+$ such that the contact Hamiltonian $h+f$ has no 1-periodic orbits.
\end{defn}

The next proposition implies that $\Pi(h)$ is non-empty for $h\in\mathcal{H}_\Omega(\Sigma)$. It is also used in the proof of Lemma~\ref{lem:invlimstab} below.

\begin{prop}\label{prop:no1open}
Let $\Sigma$ be a closed contact manifold with a contact form. Let $h:\Sigma\to\R$ be a contact Hamiltonian such that $h$ has no non-constant 1-periodic orbits, and such that $dY^h(p)=0$ for all $p\in\Sigma$ at which the vector field $Y^h$ vanishes. Then, there exists a $C^2$ neighbourhood of $h$ in $C^\infty(\Sigma)$ such that the flow of $g$ has no  non-constant 1-periodic orbits for all $g$ in that neighbourhood.
\end{prop}
\begin{proof}
Assume the contrary. Then, there exist a sequence of contact Hamiltonians $h_k$ and a sequence $x_k\in\Sigma$ such that $h_k\to h$ in $C^2$ topology, such that $x_k\to x_0,$ and such that $t\mapsto \varphi_t^{h_k}(x_k)$ is a non-constant 1-periodic orbit of $h_k.$ This implies that $t\mapsto \varphi_t^h(x_0)$ is a 1-periodic orbit of $h,$ and therefore, has to be constant.  By assumptions, $dY^h(x_0)=0.$ The map
$C^\infty(\Sigma)\to\mathfrak{X}(\Sigma)$
that assigns the contact vector field to a contact Hamiltonian is continuous with respect to $C^2$ topology on $C^\infty(\Sigma)$ and $C^1$ topology on $\mathfrak{X}(\Sigma)$. Consequently (since $h_k\to h$ in $C^2$ topology), $Y^{h_k}\to Y^h$ in $C^1$ topology. Therefore, for each $L>0,$ there exists a neighbourhood $U\subset \Sigma$ of $x_0$ and $N\in\mathbb{N}$ such that $\left. Y^{h_k}\right|_{U}$ is Lipschitz with Lipschitz constant $L$ for all $k\geqslant N.$ For $k$ big enough, the loop $t\mapsto \varphi_t^{h_k}(x_k)$ is contained in the neighbourhood $U.$ This contradicts \cite{yorke1969periods} because for $L$ small enough there are no non-constant 1-periodic orbits of $h_k$ in $U.$
\end{proof}

The following definition introduces the selective symplectic homology.
\begin{defn}\label{def:SSH}
Let $W$ be a Liouville domain, and let $\Omega\subset \partial W$ be an open subset of the boundary $\Sigma:=\partial W.$ The \emph{selective symplectic homology} with respect to $\Omega$ is defined to be
\[ SH_\ast^\Omega(W):=\underset{h\in\mathcal{H}_\Omega(\Sigma)}{\lim_{\longrightarrow}}\:\:\underset{f\in\Pi(h)}{\lim_{\longleftarrow}}\: HF_\ast(h+f). \]
The limits are taken with respect to the continuation maps.
\end{defn}

Given $h\in\mathcal{H}_\Omega(\Sigma),$ Proposition~\ref{prop:no1open} implies that for $f:\Sigma\to\R^+$ smooth and small enough (with respect to the $C^2$ topology), the contact Hamiltonian $h+f$ has no 1-periodic orbits. As a consequence, the groups $HF_\ast(h+f_1)$ and $HF_\ast(h+f_2)$ are canonically isomorphic for $f_1$ and $f_2$ sufficiently small. In other words, the inverse limit
\[\underset{f\in\Pi(h)}{\lim_{\longleftarrow}} HF_\ast (h+f)\]
stabilizes for $h\in\mathcal{H}_\Omega(W)$. This is proven in the next lemma. 

\begin{lem}\label{lem:invlimstab}
Let $W$ be a Liouville domain, let $\Omega\subset \partial W$ be an open subset, and let $h\in\mathcal{H}_\Omega(W)$. Then, there exists an open convex neighbourhood $U$ of 0 (seen as a constant function on $\partial W$) in $C^2$ topology such that the natural map
\[\underset{f\in\Pi(h)}{\lim_{\longleftarrow}} HF_\ast (h+f) \to HF_\ast(h+g) \]
is an isomorphism for all $g\in  C^\infty(\partial W, \R^+)\cap U$. 
\end{lem}
\begin{proof}
Proposition~\ref{prop:no1open} implies that there exists a convex $C^2$ neighbourhood $U$ of the constant function $\partial W\to \R: p\mapsto 0$ such that $h+ f$ has no non-constant 1-periodic orbits if $f\in U$. Since $h+f$ is positive for a positive function $f\in U$, it does not have any constant orbits either (the corresponding vector field is nowhere 0). Hence, $h+f$ has no 1-periodic orbits for all positive functions $f:\partial W\to \R^+$ from $U$. This, in particular, implies $ \mathcal{O}:=C^\infty(\partial W, \R^+)\cap U \subset \Pi(h).$
The set $\mathcal{O}$ is also convex. Therefore, 
$(1-s)\cdot f_a + s\cdot f_b\in\mathcal{O}$
for all $f_a, f_b\in\mathcal{O}$ and $s\in[0,1]$. If, additionally, $f_a\leqslant f_b$, then  $h+ (1-s)\cdot f_a + s\cdot f_b$ is an increasing family (in $s$-variable) of admissible contact Hamiltonians. Theorem~1.3 from \cite{uljarevic2022hamiltonian} asserts that the continuation map
$HF_\ast(h+f_a)\to HF_\ast(h+f_b)$
is an isomorphism in this case. This implies the claim of the lemma.
\end{proof}

The set $U$ from Lemma~\ref{lem:invlimstab} is not unique. For technical reasons, it is useful to choose one specific such set (we will denote it by $\mathcal{U}(h)$)\label{p:U} for a given contact Hamiltonian $h\in\mathcal{H}_\Omega(\partial W)$. The construction of $\mathcal{U}(h)$ follows. Let $\psi_j: V_j\to\partial W$ be charts on $\partial W$ and let $K_j\subset \psi(V_j)$ be compact subsets, $j\in\{1,\ldots, m\}$, such that $\bigcup_{j=1}^m K_j=\partial W$. Denote by $\norm{\cdot}_{C^2}$ the norm on $C^\infty(\partial W, \R)$ defined by
\[\norm{f}_{C^2}:= \underset{i\in\{0,1,2\}}{\max_{j\in\{1,\ldots, m\}}}\max_{K_j} \norm{D^i(f\circ\psi_j)}. \]
The norm $\norm{\cdot}_{C^2}$ induces the $C^2$ topology on $C^\infty(\partial W, \R)$. Denote by $\mathcal{B}(\varrho)\subset C^\infty(\partial W, \R)$ the open ball with respect to $\norm{\cdot}_{C^2}$ centered at 0 of radius $\varrho$. Define $\mathcal{U}(h)$ as the union of the balls $\mathcal{B}(\varrho)$ that have the following property: the contact Hamiltonian $h+f$ has no non-constant 1-periodic orbits for all $f\in\mathcal{B}(\varrho)$. The set $\mathcal{U}(h)$ is open as the union of open subsets. It is convex as the union of nested convex sets. And, it is non-empty by Proposition~\ref{prop:no1open}. The subset of $\mathcal{U}(h)$ consisting of strictly positive functions is denoted by $\mathcal{O}(h)$, i.e.
$\mathcal{O}(h):= \mathcal{U}(h)\cap C^\infty(\partial W, \R).$\label{p:O}

\section{Behaviour under direct limits}

The next theorem claims that the selective symplectic homology behaves well with respect to direct limits.

\begin{theorem}\label{thm:limitsh}
Let $(W,\lambda)$ be a Liouville domain, and let $\Omega_1,\Omega_2,\ldots$ be a sequence of open subsets of $\partial W$ such that $\Omega_k\subset \Omega_{k+1}$ for all $k\in\mathbb{N}.$ Denote $\Omega:=\bigcup_{k}\Omega_k.$ Then, the map
\begin{align*}
    & \mathfrak{P} : \lim_{k\to +\infty} SH_\ast^{\Omega_k}(W)\to SH_\ast^\Omega(W),
\end{align*}
furnished by continuation maps, is an isomorphism.
\end{theorem}
\begin{proof}
Let $h$ be an arbitrary contact Hamiltonian in $\mathcal{H}_\Omega(\partial W)$. Since $\op{supp} h$ is a compact subset of $\Omega$, and since $\bigcup\Omega_k=\Omega$, there exists $k\in\mathbb{N}$ such that $\op{supp} h\subset \Omega_k$. For such a $k$, we have $h\in\mathcal{H}_{\Omega_k}(\partial W)$. In other words, $\bigcup_k \mathcal{H}_{\Omega_k}(\partial W)= \mathcal{H}_\Omega(\partial W)$. The theorem now follows from the next abstract lemma.
\end{proof}

The following lemma was used in the proof of Theorem~\ref{thm:limitsh}.

\begin{lem}
    Let $(P,\leqslant)$ be a directed set and let $P_1\subset P_2\subset\cdots\subset P$ be subsets of $P$ such that $(P_j,\leqslant)$ is a directed set for all $j\in \mathbb{N}$, and such that $\bigcup_j P_j= P$. Let $\{G_a\}_{a\in P}$ be a directed system over $P$. Then, there exists a canonical isomorphism
    \[\underset{j}{\lim_{\longrightarrow}}\:\underset{a\in P_j}{\lim_{\longrightarrow}}\: G_a\:\to\: \underset{a\in P}{\lim_{\longrightarrow}}\: G_a.\]
\end{lem}
\begin{proof}
    Denote by $f_a^b:G_a\to G_b$, $a\leqslant b$ the morphisms of the directed system $\{G_a\}$. Denote by 
    \[\phi_a^j: G_a\to\underset{b\in P_j}{\lim_{\longrightarrow}} G_b\] 
    the canonical map, defined if $a\in P_j$. Since $\phi_b^j\circ f_a^b=\phi_a^j$ whenever $a\leqslant b$ and $a,b\in P_j$, the morphisms $\{\phi_a^j\}_{a\in P_i}$ induce a morphism
    \[F_i^j: \underset{a\in P_i}{\lim_{\longrightarrow}} G_a \to \underset{a\in P_j}{\lim_{\longrightarrow}} G_a \]
    for positive integers $i\leqslant j$. The morphisms $\{F_i^j\}_{i\leqslant j}$ make $\displaystyle \left\{\underset{a\in P_j}{\lim_{\longrightarrow}} G_a\right\}_{j\in\mathbb{N}}$ into a directed system indexed by $(\mathbb{N}, \leqslant)$. Denote by
    \[\Phi_j: \underset{a\in P_j}{\lim_{\longrightarrow}} G_a \to \underset{j\in \mathbb{N}}{\lim_{\longrightarrow}}\: \underset{a\in P_j}{\lim_{\longrightarrow}} G_a\]
    the canonical map. 

    We will prove the lemma by showing that $\displaystyle \underset{j\in \mathbb{N}}{\lim_{\longrightarrow}} \underset{a\in P_j}{\lim_{\longrightarrow}} G_a$ together with the maps $\Phi_j\circ\phi_a^j$, $a\in P$ satisfies the universal property of the direct limit. Let $\left(Y, \{\psi_a\}_{a\in P}\right)$ be a target, i.e. $\{\psi_a: G_a\to Y\}_a$ is a collection of morphisms that satisfy $\psi_b\circ f_a^b=\psi_a$ for all $a,b\in P$ such that $a\leqslant b$.
    Since $\left(Y, \{\psi_a\}_{a\in P_j}\right)$ is a target for the directed system $\{G_a\}_{a\in P_j}$, the universal property of the direct limit implies that there exists a unique morphism
    \[\Psi_j: \underset{a\in P_j}{\lim_{\longrightarrow}} G_a\to Y\]
    such that $\Psi_j\circ \phi_a^j= \psi_a$ for all positive integers $i\leqslant j$. By applying the universal property again, we conclude that there exists a unique morphism
    \[\Psi : \underset{j}{\lim_{\longrightarrow}}\:\underset{a\in P_j}{\lim_{\longrightarrow}}\: G_a\to Y \]
    such that $\Psi\circ\Phi_j=\Psi_j$. Since 
    \[\Psi\circ\Phi_j\circ\phi_a^j= \Psi_j\circ \phi_a^j=\psi_a,\]
    this finishes the proof.
    
\end{proof}

\section{Conjugation isomorphisms}\label{sec:conjugationisomorphisms}

Let $(M,\lambda)$ be a Liouville domain of finite type. The group of symplectomorphisms $\psi :M\to M$ that preserve the Liouville form outside of a compact subset is denoted by $\op{Symp}^\ast(M,\lambda)$. If $M=\hat{W}$ is the completion of a Liouville domain $(W, \lambda)$, then for $\psi\in \op{Symp}^\ast(M, \lambda)$ there exist a contactomorphism $\varphi:\partial W \to\partial W$ and a positive smooth function $f:\partial W\to\R^+$ such that
\[ \psi(x,r)= (\varphi(x), r\cdot f(x)), \]
for $x\in\partial W$ and $r\in\R^+$ large enough. The contactomorphism $\varphi$ is called the \emph{ideal restriction} of $\psi$. To an element $\psi\in\op{Symp}^\ast(M, \lambda)$, one can associate isomorphisms, called \emph{conjugation isomorphisms},
\begin{align*}
    & \mathcal{C}(\psi) : HF_\ast(H,J) \to HF_\ast(\psi^\ast H, \psi^\ast J),
\end{align*}
where $(H,J)$ is regular Floer data. The isomorphisms $\mathcal{C}(\psi)$ are defined on the generators by
\[\gamma\mapsto \psi^\ast \gamma =\psi^{-1}\circ \gamma.\]
They are isomorphisms already on the chain level, and already on the chain level, they commute with the continuation maps.

\begin{prop}
Let $(M,\lambda)$ be the completion of a Liouville domain $(W, \lambda)$, let $\psi\in\op{Symp}^\ast(M,\lambda)$, and let $\varphi:\partial W\to \partial W$ be the ideal restriction of $\psi$. Then, the conjugation isomorphisms with respect to $\psi$ give rise to isomorphisms (called the same)
\begin{align*}
    &\mathcal{C}(\psi) : SH_\ast^{\Omega}(W)\to SH_\ast^{\varphi^{-1}(\Omega)}(W),
\end{align*}
for every open subset $\Omega\subset \partial W$.
\end{prop}
\begin{proof}
Let $h\in\mathcal{H}_\Omega(\partial W)$, let $f\in \Pi(h)$, and let $(H,J)$ be Floer data for $W$ and for the contact Hamiltonian $h+f$. The Floer data $(\psi^\ast H, \psi^\ast J)$ corresponds to the contact Hamiltonain $g\cdot (h+f)\circ \varphi$, where $g:\partial W\to \R^+$ is a certain positive smooth function. Moreover, $g\cdot h\circ\varphi \in \mathcal{H}_{\varphi^{-1}(\Omega)}(W)$ and $g\cdot f\circ \varphi \in \Pi(g\cdot h\circ\varphi).$ Since the conjugation isomorphisms commute with the continuation maps and since the relations above hold, the conjugation isomorphisms give rise to an isomorphism
\[\mathcal{C}(\psi) : SH_\ast^{\Omega}(W) \to SH_\ast^{\varphi^{-1}(\Omega)}(W).\]
\end{proof}

Now, the proof of Theorem~\ref{thm:conjVSsont} from the introduction follows directly.

\settheoremtag{\ref{thm:conjVSsont}}
\begin{theorem}
Let $W$ be a Liouville domain, let $\psi:\hat{W}\to\hat{W}$ be a symplectomorphism that preserves the Liouville form outside of a compact set, and let $\varphi:\partial W\to\partial W$ be the ideal restriction of $\psi$. Let $\Omega_a\subset \Omega_b\subset \partial W$ be open subsets. Then, the following diagram, consisting of conjugation isomorphisms and continuation maps, commutes
\[\begin{tikzcd}
SH_\ast^{\Omega_a}(W) \arrow{r}{\mathcal{C}(\psi)}\arrow{d}{}& SH_\ast^{\varphi^{-1}(\Omega_a)}(W)\arrow{d}{}\\
SH_\ast^{\Omega_b}(W) \arrow{r}{\mathcal{C}(\psi)}& SH_\ast^{\varphi^{-1}(\Omega_b)}(W).
\end{tikzcd}\]
\end{theorem}
\begin{proof}
The proof follows directly from the commutativity of the conjugation isomorphisms and the continuation maps on the level of $HF_\ast(H,J)$.
\end{proof}

\section{Selective symplectic homology for a Darboux chart}\label{sec:darboux}

This section proves that sufficiently small open subsets on the boundary of a Liouville domain have finite dimensional selective symplectic homology.
Let $a_1, \ldots, a_n, b\in\R^+$. The contact polydisc $P=P(a_1,\ldots, a_n, b)$ is a subset of the standard contact $\R^{2n+1}$ (endowed with the contact form $dz + \sum_{j=1}^n(x_jdy_j -y_jdx_j)$) that is given by
\[P:= \left\{ (x,y,z)\in\R^n\times\R^n\times\R\:|\: z^2\leqslant b^2\:\&\: (\forall j\in\{1,\ldots, n\})\: x_j^2+y_j^2\leqslant a_j^2 \right\}.\]

\begin{theorem}\label{thm:sshdarboux}
Let $W$ be a Liouville domain and let $P\subset \partial W$ be a contact polydisc in a Darboux chart. Then, the continuation map
\[SH_\ast^{\emptyset}(W)\to SH_\ast^{\op{int}P}(W)\]
is an isomorphism.
\end{theorem}

The next lemma is used in the proof of Theorem~\ref{thm:sshdarboux}.
\begin{lem}\label{lem:bump}
Let $\alpha := dz + \sum_{j=1}^n (x_j dy_j - y_j dx_j)$ be the standard contact form on $\R^{2n+1}$. Denote by $(r_j, \theta_j)$ polar coordinates in the $(x_j, y_j)$-plane, $j=1,\ldots, n$. Let $h:\R^{2n+1}\to [0,+\infty)$ be a contact Hamiltonian of the form
\[h(r, \theta, z):= \varepsilon + g(z)\cdot \prod_{j=1}^n f_j(r_j),\]
where $\varepsilon\in\R^+$, $g:\R\to [0,+\infty)$ is a smooth function, and $f_j:[0,+\infty)\to [0,+\infty)$ is a (not necessarily strictly) decreasing smooth function, $j=1,\ldots, n$. Then, the $z$-coordinate strictly decreases along the trajectories of the contact Hamiltonian $h$ (with respect to the contact form $\alpha$).
\end{lem}
\begin{proof}
Let $Y^h$ be the vector field of the contact Hamiltonian $h$, i.e. the vector field that satisfies $\alpha(Y^h)=- h$ and $d\alpha(Y^h, \cdot)= dh - dh(\partial_z)\cdot \alpha$. Then,
\[ dz(Y^h)= -\varepsilon + g(z)\cdot \left( -\prod_{k=1}^n f_k(r_k) +\frac{1}{2}\cdot \sum_{j=1}^n \left( r_j\cdot f'_j(r_j)\cdot \prod_{k\not=j} f_k(r_k) \right) \right). \]
In particular, $dz(Y^h(p))\leqslant -\varepsilon$ for all $p\in\R^{2n+1}$. Let $\gamma:I\to \R^{2n+1}$ be a trajectory of the contact Hamiltonian $h$. Then,
\[\frac{d}{dt}\left(z(\gamma(t)) \right)= dz(Y^h(\gamma(t)))\leqslant -\varepsilon.\]
Consequently, the function $t\mapsto z(\gamma(t))$ is strictly decreasing.
\end{proof}

\begin{proof}[Proof of Theorem~\ref{thm:sshdarboux}]
By assumptions, there exists a Darboux chart $\psi:O\to \R^{2n+1}$, $O\subset \partial W$, such that $\psi(P)= P(a_1, \ldots, a_n, b)$ for some $a_1,\ldots, a_n, b\in\R^+$. Since $P(a_1, \ldots, a_n, b)$ is compact and $\psi(O)$ open, there exist $b', a_1',\ldots, a_n'\in\R^+$ such that
\[P(a_1, \ldots, a_n, b)\subset \op{int} P(a_1', \ldots, a_n', b')\subset \psi(O).\]
In particular, $b<b'$. Denote $\varepsilon_1 := b'-b$

Let $h\in\mathcal{H}_{\op{int} P}(\partial W)$ be such that 
\begin{equation}\label{eq:productlike} h\circ \psi^{-1} (r, \theta, z) = g(z)\cdot \prod_{j=1}^n f_j(r_j)\end{equation}
for some smooth function $g:\R\to[0,+\infty)$ and some smooth decreasing functions $f_j:[0,+\infty)\to[0, +\infty)$, $j=1, \ldots, n$ such that $\op{supp} g \subset (0, b) $ and $\op{supp} f_j\subset (0, a_j)$. Let $\varepsilon_0\in\R^+$ be such that there are no closed Reeb orbits on $\partial W$ of period less than or equal to $\varepsilon_0$. 

Now, we show that the contact Hamiltonian $h+\varepsilon$ has no 1-periodic orbits if $0<\varepsilon<\min\{\varepsilon_0, \varepsilon_1\}$. This implies $\varepsilon\in \mathcal{O}(h)$ if $0<\varepsilon<\min\{\varepsilon_0, \varepsilon_1\}$. Let $\gamma:\R\to \partial W$ be a trajectory of the contact Hamiltonian $h+\varepsilon$. If $\gamma$ does not intersect $P$, then $\gamma$ is also a trajectory of the reparametrized Reeb flow $t\mapsto \varphi_{-\varepsilon\cdot t}$. Since $\varepsilon<\varepsilon_0$, this implies that $\gamma$ is not 1-periodic. Assume, now, that $\gamma$ does intersect $P$. If $\gamma$ is entirely contained in $O$, then Lemma~\ref{lem:bump} implies that $\gamma$ is not 1-periodic. If $\gamma$ is not entirely contained in $O$, then (by Lemma~\ref{lem:bump}) $\gamma$ intersects $\psi^{-1}\left( \R^{2n}\times[b, b'] \right)$. On $\psi^{-1}\left( \R^{2n}\times[b, b'] \right)$, the contact Hamiltonian $h+\varepsilon$ is equal to $\varepsilon$ and $\gamma(t)$ is equal to $\psi^{-1}(x,y, z-\varepsilon t)$ for some $(x,y,z)\in\R^{2n+1}$. In particular, $\gamma$ ``spends'' at least $\frac{b'-b}{\varepsilon}$ time passing through $\psi^{-1}\left( \R^{2n}\times[b, b'] \right)$. Since 
\[\frac{b'-b}{\varepsilon}> \frac{b'-b}{\varepsilon_1}=1,\]
$\gamma$ cannot be 1-periodic.

The same argument shows that the contact Hamiltonian $h^s:= s\cdot h+ \varepsilon$ has no 1-periodic orbits for all $s\in[0,1]$. Additionally, $\partial_sh^s\geqslant 0$. Therefore, the continuation map
\[HF_\ast(\varepsilon)=HF_\ast(h^0)\to HF_\ast(h^1)= HF_\ast(h+\varepsilon)\]
is an isomorphism \cite[Theorem~1.3]{uljarevic2022hamiltonian}. Since for every $\tilde{h}\in\mathcal{H}_{\op{int} P}(\partial W)$ there exists $h\in \mathcal{H}_{\op{int} P}(\partial W)$ of the form \eqref{eq:productlike} such that $\tilde{h}\leqslant h$, the theorem follows.
\end{proof}

\section{Immaterial transverse circles and selective symplectic homology of their complements}\label{sec:immaterial}

This section provides non-trivial examples where the selective symplectic homology is ``large''.  We start by defining \emph{immaterial} subsets of contact manifolds.

\begin{defn}
A subset $A$ of a contact manifold $\Sigma$ is called \emph{immaterial} if there exists a contractible loop $\varphi_t:\Sigma\to \Sigma$ of contactomorphisms such that its contact Hamiltonian $h_t:\Sigma\to\R$ (with respect to some contact form on $\Sigma$) is positive on $A$, i.e. such that it satisfies
\[(\forall x\in A)(\forall t\in\R)\quad h_t(x)>0.\]
\end{defn}

If a compact subset $A$ of a contact manifold $\Sigma$ is immaterial, then there exists a contractible loop of contactomorphisms on $\Sigma$ whose contact Hamiltonian is arbitrarily large on $A$. In fact, this property of a compact subset $A$ is equivalent to $A$ being immaterial.

\begin{lem}
A compact subset $A$ of a contact manifold $\Sigma$ is immaterial if, and only if, for every $a\in\R^+$ there exists a contractible loop of contactomorphisms on $\Sigma$ such that its contact Hamiltonian $h_t:\Sigma\to \R$ satisfies
\[(\forall x\in A)(\forall t\in\R)\quad h_t(x)\geqslant a.\]
\end{lem}
\begin{proof}
Let $a\in\R^+$ be an arbitrarily large positive number and let $A$ be a compact immaterial subset of a contact manifold $\Sigma$. Then, there exists a contractible loop $\varphi:\Sigma\to\Sigma$ of contactomorphisms such that its contact Hamiltonian $h_t:\Sigma\to\R$ satisfies
\[(\forall x\in A)(\forall t\in\R)\quad h_t(x)>0.\]
Denote $m:= \min_{x\in A, t\in\R} h_t(x)>0$. Let $k\in\mathbb{N}$ be such that $k\cdot m> a$. Denote by $h^k_t:\Sigma\to\R$ the contact Hamiltonian defined by \mbox{$h^k_t(x):=k\cdot h_{kt}(x)$}. The contact Hamiltonian $h^k$ furnishes a loop of contactomorphisms that is obtained by concatenating $\varphi$ to itself $k$ times. In particular, $h^k$ generates a contractible loop of contactomorphisms. By construction
\[(\forall x\in A)(\forall t\in \R)\quad h^k_t(x)\geqslant k\cdot m>a.\]
This proves one direction of the lemma. The other direction is obvious.
\end{proof}

The next lemma implies that a singleton (i.e. a set consisting of a single point) is immaterial in every contact manifold of dimension greater than 3. By continuity, every point in a contact manifold of dimension greater than 3 has an immaterial neighbourhood.

\begin{lem}\label{lem:ptnegl}
Let $\Sigma$ be a contact manifold of dimension $2n+1 > 3$. Then, there exists a contractible loop $\varphi_t:\Sigma\to \Sigma$ of contactomorphisms such that its contact Hamiltonian is positive at some point (for all times $t$). 
\end{lem}
\begin{proof}
Let $\mathbb{S}^{2n+1}$ be the standard contact sphere seen as the unit sphere in $\mathbb{C}^{n+1}$ centered at the origin. The unitary matrices act on $\mathbb{S}^{2n+1}$ as contactomorphisms. Let $\psi_t:\mathbb{S}^{2n+1}\to \mathbb{S}^{2n+1}$ be the contact circle action given by
\[ \psi_t(z):= \left( z_1, \ldots, z_{n-1}, e^{2\pi i t} z_n, e^{-2\pi i t} z_{n+1}  \right). \]
The loop
\[t\mapsto \left[\begin{matrix} e^{2\pi i t} & 0\\ 0 & e^{-2\pi i t} \end{matrix}\right]\]
is contractible in the unitary group $U(2)$. Hence, there exists a smooth $s$-family $A^s$, $s\in[0,1]$,  of loops in $U(2)$ such that
\[A^1(t)=  \left[\begin{matrix} e^{2\pi i t} & 0\\ 0 & e^{-2\pi i t} \end{matrix}\right]\]
and such that $A^0(t)= \left[\begin{matrix}1&0\\ 0&1\end{matrix}\right]$ for all $t$. Denote $\psi^s_t(z):=\left[ \begin{matrix} \mathbb{1}_{n-1} & \\ & A^s(t) \end{matrix}\right] z$. For all $s\in[0,1]$, $\psi^s$ is a loop of contactomorphisms of $\mathbb{S}^{2n+1}$ and $\psi_t^0=\op{id}$, $\psi_t^1=\psi_t$. Therefore, $\psi_t$ is a contractible loop of contactomorphisms. Denote by $h^s_t:\mathbb{S}^{2n+1}\to \mathbb{R}$ the contact Hamiltonian of $\psi^s_t$ and $h:=h^1$. 
Explicitly, $h(z_1,\ldots, z_{n+1})= 2\pi\cdot \left(\abs{z_{n+1}}^2-\abs{z_n}^2\right)$. In particular, $h$ is po\-si\-tive at the point $(0,\ldots, 0,1)$. 
Denote $V(r):=\left\{ z\in\mathbb{S}^{2n+1}\:|\: \abs{z_1}> 1-r \right\}$ and let $\varepsilon\in (0,1)$. Let $\mu: \mathbb{S}^{2n+1}\to[0,1]$ be a smooth cut-off function such that $\mu(x)=0$ for $x$ in a neighbourhood of $p:=(1,0,\ldots, 0)$ and such that $\mu(x)=1$ for $x\in\mathbb{S}^{2n+1}\setminus V(\frac{\varepsilon}{2})$. Let $f_t^s(x):= \mu(x)\cdot h^s_t(x)$. By the construction of $\mu$ and since $V(r)$ is invariant under $\psi^s_t$ for all $r,s$, and $t$, the contactomorphism $\varphi_1^{f^s}$ is compactly supported in $V(\varepsilon)$ for all $s$. Let $g^s_t:\mathbb{S}^{2n+1}\to\R$, $s\in[0,1]$ be the contact Hamiltonian that generates $t\mapsto \varphi_1^{f^{t\cdot s}}$, i.e. $\varphi_t^{g^s}= \varphi_1^{f^{t\cdot s}}$. Denote $g:=g^1.$

The map $\varphi^{f^1}_t\circ(\varphi_t^g)^{-1}$ is a loop of contactomorphisms. Its contact Hamiltonian $e_t:\mathbb{S}^{2n+1}\to\R$ is equal to 0 in a neighbourhood of $p$ and coincides with $f^1$ in $\mathbb{S}^{2n+1}\setminus V(\varepsilon)$. Consequently (since $f^1$ and $h$ coincide in $\mathbb{S}^{2n+1}\setminus V(\varepsilon)$), the contact Hamiltonians $e$ and $h$ coincide in $\mathbb{S}^{2n+1}\setminus V(\varepsilon)$. This implies that $\varphi^{f^1}_t\circ(\varphi_t^g)^{-1}$ is a loop of contactomorphisms of $\mathbb{S}^{2n+1}$ that are compactly supported in the complement of a neighbourhood of $p$. Additionally, this implies that there exists $q\in\mathbb{S}^{2n+1}\setminus V(\varepsilon)$ such that $e_t(q)=h(q)>0$ for all $t$.
The loop $\varphi_t^e=\varphi^{f^1}_t\circ(\varphi_t^g)^{-1}$ is contractible via the homotopy $\left\{\varphi^{f^s}_t\circ(\varphi_t^{g^s})^{-1}\right\}_{s\in[0,1]}$ that is also compactly supported in the complement of a neighbourhood of $p$. Since $\mathbb{S}^{2n+1}\setminus \{p\}$ is contactomorphic to the standard $\R^{2n+1}$ and since every contact manifold has a contact Darboux chart around each of its points, the lemma follows. 
\end{proof}

The following theorem implies that the complement of an immaterial circle has infinite dimensional selective symplectic homology under some additional assumptions.
\begin{theorem}\label{thm:compnegl}
Let $W$ be a Liouville domain and let $\Gamma\subset \partial W$ be an immaterial embedded circle that is transverse to the contact distribution. Denote $\Omega:=\partial W\setminus \Gamma$. Then, the rank of the continuation map $SH_\ast^\Omega(W)\to SH_\ast(W)$ is equal to $\dim SH_\ast(W)$.
\end{theorem}
\begin{proof}
This proof assumes results of Section~\ref{sec:pathiso}.
For an admissible contact Hamiltonian $h_t:\partial W\to \R$, denote by $r(h)=r(W, h)$ the rank of the canonical map $HF_\ast(h)\to SH_\ast(W)$. It is enough to prove that for every admissible $\ell\in\R$ there exists $h\in\mathcal{H}_\Omega(\partial W)$ and $\varepsilon\in\mathcal{O}(h)$ such that $r(\ell)\leqslant r(h+\varepsilon)$. Denote by $\alpha$ the contact form on $\partial W$ (the restriction of the Liouville form). Without loss of generality (see Theorem~2.5.15 and Example~2.5.16 in \cite{geiges2008introduction}), we may assume that there exists an open neighbourhood $U\subset \partial W$ of $\Gamma$ and an embedding $\psi: U\to \mathbb{C}^n\times\mathbb{S}^1$ such that $\psi(\Gamma)= \{0\}\times\mathbb{S}^1$ and such that
\[\alpha=\psi^\ast\left( d\theta + \frac{i}{2}\sum_{j=1}^n (z_jd\overline{z}_j-\overline{z}_jdz_j)\right).\]
Here, $z=(z_1,\ldots, z_n)\in\mathbb{C}^n$ and $\theta\in\mathbb{S}^1$. Let $\ell\in\R$ be an arbitrary admissible (constant) slope. Since $\Gamma$ is immaterial, there exists a contractible loop of contactomorphisms $\varphi^f_t:\partial W\to\partial W$ (which we see as a 1-periodic $\R$-family of contactomorphisms) such that its contact Hamiltonian $f_t:\partial W\to\R$ satisfies $\min_{x\in\Gamma, t\in\R} f_t(x)\geqslant 2\ell$.
Denote $m:=\min_{x\in\partial W, t\in\R} f_t(x)$. Let $h\in\mathcal{H}_{\Omega}(\partial W)$ be a strict contact Hamiltonian (i.e. its flow preserves the contact form $\alpha$ ) such that $h(x)\geqslant \ell- m$ for $x$ in the set $ \left\{ x\in\partial W\:|\: \min_{t\in\R} f_t(x)\leqslant \ell \right\}.$ The contact Hamiltonian $h$ can be constructed as follows.
Since the function $x\mapsto\min_{t\in\R} f_t(x)$ is continuous, the set $S:=\{x\in\partial W\:|\: \min_{t\in\R} f_t(x)\leqslant\ell\}$ is closed. Therefore, there exists a ball $B(r)\subset \mathbb{C}^n$ centered at the origin with sufficiently small radius $r$ such that $\overline{B(r)}\times\mathbb{S}^1\subset \psi(\partial W \setminus S)$. Now, we choose $h$ to be equal to a constant greater than $\ell-m$ on $\partial W\setminus \psi^{-1}\left( \overline{B(r)}\times \mathbb{S}^1 \right)$ and such that $h\circ\psi^{-1}(z, \theta)= \overline{h}(z_1^2+\cdots+ z_n^2)$ for $\abs{z}<r$ and for some smooth function $\overline{h}: [0,+\infty)\to [0,+\infty)$ that is equal to 0 near 0. Generically, $h$ has no non-constant 1-periodic orbits.

Let $\varepsilon\in\R^+$ be a sufficiently small positive number such that $\varepsilon\in\mathcal{O}(h)$ and denote $h^\varepsilon:= h+\varepsilon.$ Let $g:=h^\varepsilon\# f$ be the contact Hamiltonian that generates the contact isotopy $\varphi_t^{h^\varepsilon}\circ\varphi_t^f$, i.e.
\[ g_t(x) := h^{\varepsilon}(x) + f_t\circ \left(\varphi^{h^\varepsilon}_t\right)^{-1}(x). \]
(In the last formula, we used that $h^\varepsilon$ is a strict contact Hamiltonian.) If $h^\varepsilon (x) < \ell-m$, then (since $h^\varepsilon$ is autonomous and strict) \mbox{$h^\varepsilon\circ\left(\varphi_t^{h^\varepsilon}\right)^{-1}(x)<\ell-m$} for all $t$. Consequently (by the choice of $h$), $\min_{s\in\R} f_s\circ \left( \varphi_t^{h^\varepsilon}\right)^{-1}(x)> \ell$. This implies $g_t(x)\geqslant \ell$ for all $x\in\partial W$ and $t\in\R$.

Since $\varphi^f$ is a contractible loop of contactomorphisms, there exists a smooth homotopy $\sigma$ from the constant loop $t\mapsto \op{id}$ to $\varphi^f$. By Section~\ref{sec:pathiso}, there exist isomorphisms $\mathcal{B}(\varphi^f,\sigma) : HF_\ast(h^\varepsilon)\to HF_\ast(h^\varepsilon\# f)$ and $\mathcal{B}(\varphi^f,\sigma) : SH_\ast(W)\to SH_\ast(W)$ such that the following diagram
\[\begin{tikzcd}
    SH_\ast(W) \arrow{r}{\mathcal{B}(\varphi^f,\sigma)}& SH_\ast(W)\\
    HF_\ast(h^\varepsilon) \arrow{u}\arrow{r}{\mathcal{B}(\varphi^f,\sigma)}& HF_\ast(h^\varepsilon\#f),\arrow{u}
\end{tikzcd}\]
whose vertical arrows represent the continuation maps, commutes. Consequently, $r(h^\varepsilon)=r(h^\varepsilon\# f)= r(g)$. Since $g\geqslant \ell$, we have $r(g)\geqslant r(\ell).$ This further implies $r(h+\varepsilon)= r(h^\varepsilon)\geqslant r(\ell)$ and the proof is finished.

\end{proof}

\section{Applications to contact non-squeezing}\label{sec:main}

We start with a proof of Theorem~\ref{thm:ranknonsqueezing} from the introduction.
\settheoremtag{\ref{thm:ranknonsqueezing}}
\begin{theorem}
Let $W$ be a Liouville domain and let $\Omega_a, \Omega_b\subset \partial W$ be open subsets. If the rank of the continuation map
$SH_\ast^{\Omega_b}(W)\to SH_\ast(W)$
is (strictly) greater than the rank of the continuation map
$SH_\ast^{\Omega_a}(W)\to SH_\ast(W),$
then $\Omega_b$ cannot be contactly squeezed into $\Omega_a$.
\end{theorem}
\begin{proof}
Denote by $r(\Omega)\in\mathbb{N}\cup\{0,\infty\}$ the rank of the continuation map
$SH_\ast^\Omega(W)\to SH_\ast(W).$
Assume the contrary, i.e. that there exist open subsets $\Omega_a, \Omega_b\subset\partial W$ with $r(\Omega_a)<r(\Omega_b)$ and a contact isotopy $\varphi_t:\partial W\to\partial W$, $t\in[0,1]$ such that $\varphi_0=\op{id}$ and such that $\varphi_1(\Omega_b)\subset\Omega_a$. By Section~\ref{sec:conjugationiso} (and Section~\ref{sec:conjugationisomorphisms}), $r(\varphi_1(\Omega_b))= r(\Omega_b)$. Since $\varphi_1(\Omega_b)\subset\Omega_a$, the continuation map
$SH_\ast^{\varphi_1(\Omega_b)}(W)\to SH_\ast(W)$
factors through the continuation map
$SH_\ast^{\Omega_a}(W)\to SH_\ast(W).$
Hence, 
$r(\Omega_b)= r(\varphi_1(\Omega_b))\leqslant r(\Omega_a).$
This contradicts the assumption $r(\Omega_a)<r(\Omega_b)$.
\end{proof}

\settheoremtag{\ref{thm:homologyspheres}}
\begin{theorem}
Let $n > 2$ be a natural number and let $W$ be a $2n$-dimensional Liouville domain such that $\dim SH_\ast(W)> \sum_{j=0}^{2n} \dim H_j(W;\mathbb{Z}_2)$ and such that $\partial W$ is a homotopy sphere. Then, there exist two embedded closed balls $B_1, B_2\subset \partial W$ of dimension $2n-1$ such that $B_1$ cannot be contactly squeezed into $B_2$.
\end{theorem}
\begin{proof}
Denote $\Sigma:=\partial W$, and denote by $r(\Omega)\in\mathbb{N}\cup\{0,\infty\}$ the rank of the continuation map
$SH_\ast^{\Omega}(W)\to SH_\ast(W)$
for an open subset $\Omega\subset\Sigma$.

\textbf{Step~1} (A subset with a small rank). Since $SH_\ast^\emptyset(W)$ is isomorphic to $H_\ast(W, \partial W; \mathbb{Z}_2)$, Theorem~\ref{thm:sshdarboux} implies that there exists a non-empty open subset $\Omega\subset \Sigma$ such that $r(\Omega)\leqslant \sum_{j=1}^{2n} H_j(W,\partial W; \mathbb{Z}_2)= \sum_{j=0}^{2n} H_j(W;\mathbb{Z}_2)$.

\textbf{Step~2} (A subset with a large rank). This step proves that for every $c\in \R$ with $c\leqslant \dim SH_\ast(W)$, there exists an open non-dense subset $U\subset\Sigma$ such that $r(U)\geqslant c$. Assume the contrary, i.e. that there exists $c\in\R$ such that $r(U)< c$ for every open non-dense subset $U\subset \Sigma$. By Lemma~\ref{lem:ptnegl}, there exists a sufficiently small contact Darboux chart on $\Sigma$ that is immaterial. Let $\Gamma$ be  an embedded circle in that chart that is transverse to the contact distribution. Then, $\Gamma$ is immaterial as well. Consequently (by Theorem~\ref{thm:compnegl}), the continuation map $\Phi: SH_\ast^{\Sigma\setminus\Gamma}(W)\to SH_\ast(W)$ has rank equal to $\dim SH_\ast(W)$, i.e. $r(\Sigma\setminus\Gamma)=\dim SH_\ast(W)$.  
Hence, there exist $a_1, \ldots, a_k\in SH_\ast^{\Sigma\setminus\Gamma}(W)$, with $k\geqslant c$, such that $\Phi(a_1), \ldots, \Phi(a_k)$ are linearly independent.
Let $U_1, U_2, \ldots $ be an increasing family of open non-dense subsets of $\Sigma$ such that $\bigcup_{j=1}^\infty U_j=\Sigma\setminus\Gamma$. Since, by Theorem~\ref{thm:limitsh}, continuation maps furnish an isomorphism  
\[\underset{j}{\lim_{\longrightarrow}} SH_\ast^{U_j}(W)\to SH_\ast^{\Sigma\setminus\Gamma}(W),\]
there exist $m\in\mathbb{N}$ and $b_1,\ldots, b_k \in SH^{U_m}_\ast(W)$ such that $b_1, \ldots, b_k$ are mapped to $a_1, \ldots, a_k$ via the continuation map $SH_\ast^{U_m}(W)\to SH_\ast^{\Sigma\setminus\Gamma}(W).$
The images of $b_1, \ldots, b_k$ under the continuation map
$SH_\ast^{U_m}(W)\to SH_\ast(W)$
are equal to $\Phi(a_1), \ldots, \Phi(a_k)$, and therefore, are linearly independent. Hence, $r(U_m)\geqslant k \geqslant c.$ This contradicts the assumption and finishes Step~2.

\textbf{Step~3} (The final details). This step finishes the proof. Let $\Omega, U\subset \Sigma$ be open non-empty subsets such that $U$ is non-dense, and such that $r(\Omega)<r(U)$. Step~1 and Step~2 prove the existence of such sets $\Omega$ and $U$. Let $a\in \Omega$ and $b\in \Sigma\setminus\{a\}\setminus\overline{U}$ be two points. Since $\Sigma$ is a homotopy sphere, there exists a Morse function $f:\Sigma\to \R$ that attains its minimum at $a$ and its maximum at $b$ and that has no other critical points. The existence of such a function $f$ is guaranteed by the results of Smale \cite{smale1956generalized,smale1962structure,smale1962structure5} and Cerf \cite{cerf1968diffeomorphismes} (see also \cite[Proposition~2.2]{saeki2006morse}). The Morse theory implies that
$\Sigma_t:= f^{-1} \big( (-\infty, t]\big)$
is the standard $(2n-1)$-dimensional closed ball smoothly embedded into $\Sigma$ for all $t\in (f(a), f(b))$ (see, for instance, \cite{banyaga2013lectures}). For $s\in (f(a), f(b))$ sufficiently close to $f(a)$, $\Sigma_s\subset \Omega$. Similarly, for $\ell\in (f(a), f(b))$ sufficiently close to $f(b)$, $\op{int} \Sigma_\ell\supset U$. Since $r(\Omega)< r(U)$, by Theorem~\ref{thm:ranknonsqueezing}, $U$ cannot be contactly squeezed into $\Omega$. Hence, $\Sigma_\ell$ cannot be contactly squeezed into $\Sigma_s$ and the proof is finished (one can take $B_2:=\Sigma_s$ and $B_1:=\Sigma_\ell$).
\end{proof}

Now, we prove the contact non-squeezing for the Ustilovsky spheres.
\settheoremtag{\ref{thm:Ustilovskyspheres}}
\begin{theorem}
Let $\Sigma$ be an Ustilovsky sphere. Then, there exist two embedded closed balls $B_1, B_2\subset \Sigma$ of dimension equal to $\dim \Sigma$ such that $B_1$ cannot be contactly squeezed into $B_2$. 
\end{theorem}
\begin{proof}
In  the view of Theorem~\ref{thm:homologyspheres}, it is enough to show that the symplectic homology of the Brieskorn variety
\[ W:=\left\{ z=(z_0,\ldots, z_{2m+1})\in\mathbb{C}^{2m+2}\:|\: z_0^p + z_1^2 +\cdots + z_{2m+1}^2=\varepsilon\:\&\: \abs{z}\leqslant1 \right\}\]
is infinite demensional. Here, $m,p\in\mathbb{N}$ are natural numbers with $p\equiv \pm 1\pmod{8}$ and $\varepsilon\in\R^+$ is sufficiently small. The Brieskorn variety $W$ is a Liouville domain whose boundary is contactomorphic to an Ustilovsky sphere, and every Ustilovsky sphere is contactomorphic to $\partial W$ for some choice of $m$ and $p$.

Let $\Sigma_k$ be the sequence of manifolds such that $\Sigma_k=\partial W$ if $p\mid k$ (we use this notation for ``$k$ is divisible by $p$'') and such that
\[\Sigma_k:=\left\{z=(z_1,\ldots, z_{2m+1})\in\mathbb{C}^{2m+1}\:|\: z_1^2+\cdots z_{2m+1}^2=0\:\&\: \abs{z}=1\right\}\]
if $p\nmid k$. Denote
\[s_k:= \frac{k}{p}\cdot \left( (4m - 2)\cdot p + 4 \right) - 2m\]
if $p\mid k$, and 
\[s_k:= (4m-2)\cdot k + 2\cdot \left\lfloor\frac{2k}{p}\right\rfloor -2m +2\]
otherwise. Theorem~B.11 from \cite{kwon2016brieskorn} implies that there exists a spectral sequence $E_{k,\ell}^r$ that converges to $SH_\ast(W)$ such that its first page is given by
\[ E^1_{k,\ell}=\left\{ \begin{matrix} H_{k+\ell-s_k}(\Sigma_k;\mathbb{Z}_2) &\text{if } k>0,\\ H_{\ell+ 2m+1}(W, \partial W; \mathbb{Z}_2) &\text{if } k=0,\\ 0 &\text{if } k<0. \end{matrix}\right. \]
The terms of $E^1$ can be explicitly computed. If $p\nmid k$, then $\Sigma_k$ is diffeomorphic to the unit cotangent bundle $S^\ast\mathbb{S}^{2m}$ of the sphere. Otherwise, $\Sigma_k$ is diffeomorphic to $\mathbb{S}^{4m+1}$, because $p\equiv \pm 1\pmod{8}$ \cite{brieskorn1966beispiele}. Therefore,
\[H_j(\Sigma_k; \mathbb{Z}_2)\cong \left\{ \begin{matrix}\mathbb{Z}_2 & \text{if } j\in\{0, 2m-1, 2m, 4m-1\},\\  0 & \text{otherwise}  \end{matrix}\right.\]
if $p\nmid k$, and 
\[H_j(\Sigma_k; \mathbb{Z}_2)\cong \left\{ \begin{matrix}\mathbb{Z}_2 & \text{if } j\in\{0,4m+1\}, \\ 0 & \text{otherwise}  \end{matrix}\right.\]
if $p\mid k$. The Brieskorn variety $W$ is homotopy equivalent to the bouquet of $p-1$ spheres of dimension $2m+1$ \cite[Theorem~6.5]{milnor2016singular}. Therefore,
\[ H_j(W, \partial W; \mathbb{Z}_2) \cong \left\{\begin{matrix} \mathbb{Z}_2 & \text{if } j=4m+2,\\ \mathbb{Z}_2^{p-1} & \text{if } j=2m+1,\\ 0 & \text{otherwise.}\end{matrix} \right. \]
Figure~\ref{fig:spectralsequence} on page~\pageref{fig:spectralsequence} shows the page $E^1_{k,\ell}$ for $p=7$ and $m=1.$ For $k\in\mathbb{N}\cup \{0\}$, denote by $Q(k)$ the unique number in $\mathbb{Z}$ such that $E^{1}_{k, Q(k)}\not=0$ and such that $E^1_{k,\ell}=0$ for all $\ell>Q(k)$. Similarly, for $k\in\mathbb{N}\cup \{0\}$, denote by $q(k)$ the unique number in $\mathbb{Z}$ such that $E^{1}_{k, q(k)}\not=0$ and such that $E^1_{k,\ell}=0$ for all $\ell<q(k)$. Explicitly, 
\begin{align*}
    & Q(k)=\left\{\begin{matrix} 2m+1 &\text{if } k=0,\\ 4m-1+s_k-k & \text{if } p\nmid k,\\ 4m+1+s_k-k & \text{if }p\mid k, \end{matrix}\right.\\
    & q(k)=\left\{\begin{matrix} 0 & \text{if } k=0,\\ s_k-k & \text{for } k\in\mathbb{N}.  \end{matrix}\right.
\end{align*}
We will show that the element in $E^1_{ap-1, Q(ap-1)}$ ``survives'' in $SH_\ast(W)$ for $a\in\mathbb{N}$. (In Figure~\ref{fig:spectralsequence}, the fields $(ap-1, Q(ap-1))$, $a=1,2$ are emphasized by thicker edges.) Both sequences $Q(k)$ and $q(k)$ are strictly increasing in $k$. Since $Q(k)$ is strictly increasing, the element in $E^1_{ap-1, Q(ap-1)}$ cannot be ``killed'' by an element from $E^1_{k,\ell}$ if $k\leqslant ap-1$. Since
\[q(ap+1)-Q(ap-1)= 4m-3\geqslant 1,\]
the element in $E^1_{ap-1, Q(ap-1)}$ cannot be ``killed'' by an element from $E^1_{k,\ell}$ if $k\geqslant ap+1$. Finally, since $Q(ap-1)-q(ap)=2$, the non-zero groups $E_{ap, \ell}^1$ are the ones for $\ell= Q(ap-1)-2$ and $\ell=Q(ap-1)+4m-1$. In particular, $E^1_{ap, Q(ap-1)}=0$. Therefore, $E^\infty_{ap-1, Q(ap-1)}\not=0$ for all $a\in\mathbb{N}$. This implies $\dim SH_\ast(W)=\infty$ and the proof is finished.
\end{proof}

\begin{figure}
    \centering
    \begin{tikzpicture}[scale=0.5]
        \foreach \i in {0,1,..., 16}{ 
            \draw[very thin, gray] (\i, 0)--(\i, 27);
            }
        
        \foreach \i in {0,...,27}{
            \draw[very thin, gray] (0,\i)--(16, \i);
            }
        
        \node[below] at (0.5, 0) {$0$};
        \node[below] at (7.5, 0) {$7$};
        \node[below] at (14.5, 0) {$14$};
        
        \node[left] at (0, 0.5) {$0$};
        \node[left] at (0, 10.5) {$10$};
        \node[left] at (0, 20.5) {$20$};
        
        \node at (0.5, 0.5) {6};
        \node at (0.5, 3.5) {1};
        
        \node at (1.5, 1.5) {1};
        \node at (1.5, 2.5) {1};
        \node at (1.5, 3.5) {1};
        \node at (1.5, 4.5) {1};
        
        \node at (2.5, 2.5) {1};
        \node at (2.5, 3.5) {1};
        \node at (2.5, 4.5) {1};
        \node at (2.5, 5.5) {1};
        
        \node at (3.5, 3.5) {1};
        \node at (3.5, 4.5) {1};
        \node at (3.5, 5.5) {1};
        \node at (3.5, 6.5) {1};
        
        \node at (4.5, 6.5) {1};
        \node at (4.5, 7.5) {1};
        \node at (4.5, 8.5) {1};
        \node at (4.5, 9.5) {1};
        
        \node at (5.5, 7.5) {1};
        \node at (5.5, 8.5) {1};
        \node at (5.5, 9.5) {1};
        \node at (5.5, 10.5) {1};
        
        \node at (6.5, 8.5) {1};
        \node at (6.5, 9.5) {1};
        \node at (6.5, 10.5) {1};
        \node at (6.5, 11.5) {1};
        
        \node at (7.5, 9.5) {1};
        \node at (7.5, 14.5) {1};
        
        \node at (8.5, 12.5) {1};
        \node at (8.5, 13.5) {1};
        \node at (8.5, 14.5) {1};
        \node at (8.5, 15.5) {1};
        
        \node at (9.5, 13.5) {1};
        \node at (9.5, 14.5) {1};
        \node at (9.5, 15.5) {1};
        \node at (9.5, 16.5) {1};
        
        \node at (10.5, 14.5) {1};
        \node at (10.5, 15.5) {1};
        \node at (10.5, 16.5) {1};
        \node at (10.5, 17.5) {1};
        
        \node at (11.5, 17.5) {1};
        \node at (11.5, 18.5) {1};
        \node at (11.5, 19.5) {1};
        \node at (11.5, 20.5) {1};
        
        \node at (12.5, 18.5) {1};
        \node at (12.5, 19.5) {1};
        \node at (12.5, 20.5) {1};
        \node at (12.5, 21.5) {1};
        
        \node at (13.5, 19.5) {1};
        \node at (13.5, 20.5) {1};
        \node at (13.5, 21.5) {1};
        \node at (13.5, 22.5) {1};
        
        \node at (14.5, 20.5) {1};
        \node at (14.5, 25.5) {1};
        
        \node at (15.5, 23.5) {1};
        \node at (15.5, 24.5) {1};
        \node at (15.5, 25.5) {1};
        \node at (15.5, 26.5) {1};
        
        \draw[very thick] (6,11)--(7,11)--(7,12)--(6,12)--(6,11);
        \draw[very thick] (13,22)--(14,22)--(14,23)--(13,23)--(13, 22);
    
    \end{tikzpicture}
    \caption{The first page of the spectral sequence from the proof of Theorem~\ref{thm:Ustilovskyspheres} for $p=7$ and $m=1$. The number in the field $(k,\ell)$ represents $\dim E^1_{k,\ell}$. Empty fields are assumed to contain zeros. }
    \label{fig:spectralsequence}
\end{figure}

Finally, we prove Corollary~\ref{cor:nonsqR}.
\settheoremtag{\ref{cor:nonsqR}}
\begin{cor}
Let $m\in\mathbb{N}$. Then, there exist a contact structure $\xi$ on $\R^{4m+1}$ and an embedded closed ball $B\subset \R^{4m+1}$ of dimension $4m+1$ such that $B$ cannot be contactly squeezed into an arbitrary non-empty open subset by a compactly supported contact isotopy of $\left(\R^{4m+1}, \xi\right)$. 
\end{cor}
\begin{proof}
Let $S$ be a $(4m+1)$-dimensional Ustilovsky sphere and let $B_1, B_2\subset S$ be two embedded closed balls such that $B_1\cup B_2\not= S$ and such that $B_2$ cannot be contactly squeezed into $B_1.$ Let $p\in S\setminus (B_1\cup B_2)$. The Ustilovsky sphere $S$ is diffeomorphic to the standard sphere. Hence, there exists a diffeomorphism  $\psi: \R^{4m+1}\to S\setminus\{p\}.$ Let $\xi$ be the pull back of the contact structure on $S$ via $\psi$. Now, we prove that $\xi$ and $B:=\psi^{-1}(B_2)$ satisfy the conditions of the corollary. Assume the contrary. Then, there exists a compactly supported contact isotopy $\varphi_t: \R^{4m+1}\to\R^{4m+1}$ such that $\varphi_0=\op{id}$ and such that $\varphi_1(B)\subset\op{int} \left(\psi^{-1}(B_1)\right)$. Since $\varphi$ is a compactly supported isotopy, $\psi\circ\varphi_t\circ\psi^{-1}$ extends to a contact isotopy on $S$. This contact isotopy squeezes $B_2$ into $B_1$. This is a contradiction that finishes the proof.
\end{proof}


\section{Isomorphisms furnished by paths of admissible contact Hamiltonians}\label{sec:pathiso}

In this section, we construct an isomorphism
\[\mathcal{B}(\{h^a\}): HF_\ast(h^0)\to HF_\ast(h^1)  \]
associated with a smooth family $h^a_t:\partial W\to \R, a\in[0,1]$ of admissible contact Hamiltonians on the boundary of a Liouville domain $W$. As a particular instance of this construction, we associate an isomorphism
\[\mathcal{B}(\varphi^f, \sigma) : HF_\ast(h)\to HF_\ast(h\#f)\]
with a contractible loop $\varphi^f_t:\partial W\to\partial W$ of contactomorphisms and a homotopy $\sigma$ from the constant loop $t\mapsto\op{id}$ to $\varphi^f.$

Denote by $\mathfrak{S}$ the set of admissible contact Hamiltonians $h_t:\partial W\to \R$ on the boundary $\partial W$ of a Liouville domain $W$. The set $\mathfrak{S}$ is open in the space of the smooth functions $\partial W\times\mathbb{S}^1\to\R$ with respect to the $C^2$ topology. Let $\norm{\cdot}_{C^2}$ be a norm inducing the $C^2$-topology. Denote by $\mathcal{F}$ the family of open balls $B_R(\eta)$ with respect to $\norm{\cdot}_{C^2}$that satisfy $B_{9R}(\eta)\subset \mathfrak{S}$. For $\mathcal{O}=B_R(\eta)\in\mathcal{F}$, denote $\tilde{\mathcal{O}}:=B_{3R}(\eta)$.

\begin{defn}
    Let $h^a, a\in[0,1]$ be a smooth family of admissible contact Hamiltonians on the boundary of a Liouville domain. The isomorphism
    \[ \mathcal{B}(\{h^a\}) : HF_\ast(h^0)\to HF_\ast(h^1)\]
    is defined in the following way:
    \begin{enumerate}
    \item Choose finitely many sets $\mathcal{O}_1, \ldots, \mathcal{O}_m\in \mathcal{F}$ such that $h^a\in \bigcup_{j=1}^m \mathcal{O}_j$ for all $a\in[0,1]$.
    \item Choose $0=a_0< a_1<\cdots< a_k=1$ such that for each $j=0,\ldots, k-1$ there exists $\ell_j\in\{1,\ldots m\}$ such that $h^a\in \mathcal{O}_{\ell_j}$ for all $a\in[a_j, a_{j+1}]$.
    \item Choose $g^0,\ldots g^{k-1}\in\mathfrak{S}$ such that $g_j\in \tilde{\mathcal{O}}_{\ell_j}$ and $g_j\leqslant h^{a_j}, h^{a_{j+1}}$ for $j=0,\ldots, k-1$.
    \item Denote by 
    \begin{align*}
        & \Phi^j : HF_\ast(g^j)\to HF_\ast(h^{a_j}) \\
        & \Psi^j : HF_\ast(g^j)\to HF_\ast(h^{a_{j+1}})
    \end{align*}
    the continuation maps (they are isomorphisms because the contact Hamiltonians $(1-s)\cdot g^j + s\cdot h^{a_j} $ and $(1-s)\cdot g^j + s\cdot h^{a_{j+1}} $ are admissible and increasing with respect to the $s$-variable, see \cite[Theorem~1.3]{uljarevic2022hamiltonian} ), $j=0,\ldots, k-1$.
    \item Define
    \[\mathcal{B}(h):= \Psi^{k-1}\circ \left( \Phi^{k-1} \right)^{-1}\circ\cdots \circ \Psi^1\circ \left(\Phi^1\right)^{-1}\circ \Psi^0\circ \left(\Phi^0\right)^{-1}.\]
\end{enumerate}
\end{defn}
The isomorphism $\mathcal{B}(\{h^a\})$ does not depend on the additional choices. Moreover, if $\{h^a\}$ and $\{f^a\}$, $a\in[0,1]$ are two smooth families of admissible contact Hamiltonians such that $h^a\leqslant f^a$ for all $a\in[0,1]$, then the following diagram commutes
 \[\begin{tikzcd}
    HF_\ast(h^0) \arrow{r}{\mathcal{B}(\{h^a\})}\arrow{d}& HF_\ast(h^1)\arrow{d}\\
    HF_\ast(f^0) \arrow{r}{\mathcal{B}(\{f^a\})}& HF_\ast(f^1).
    \end{tikzcd}\]
In the diagram, the vertical arrows represent the continuation maps. Now, we associate an isomorphism $ \mathcal{B}(\varphi^f,\sigma): HF_\ast(h)\to HF_\ast(h\#f)$ with a contractible (smooth) loop $\varphi^f_t:\partial W\to\partial W$ of contactomorphisms and a smooth homotopy $\sigma^a_t:\partial W\to\partial W$ from the constant loop based at the identity to $\varphi^f$. We see the homotopy $\sigma^a_t$ as a smooth $\R\times[0,1]$-family of contactomorphisms that is 1-periodic in the $t\in\R$ variable and such that the following holds
\begin{enumerate}
    \item $ \sigma_t^0=\op{id} $ for all $t\in\R$,
    \item $\sigma_0^a=\sigma_1^a=\op{id}$ for all $a\in[0,1]$,
    \item $ \sigma_t^1=\varphi_t^f$  for all $ t\in\R$.
\end{enumerate}
For every admissible contact Hamiltonian $h\in\mathfrak{S}$, the homotopy $\sigma$ furnishes a smooth family $\eta^a:=h\#f^a$ of admissible contact Hamiltonians. Here, $f^a$ denotes the contact Hamiltonian of the contact isotopy $t\mapsto \sigma^a_t$. Define $ \mathcal{B}(\varphi^f,\sigma):= \mathcal{B}(\{\eta^a\})$. Now we show that $\sigma$ also induces an isomorphism $SH_\ast(W)\to SH_\ast(W)$ that behaves well with respect to the canonical maps $HF_\ast(h)\to SH_\ast(W)$. For a contact Hamiltonian $h$, denote $\op{osc}(h):=\max_{x,t} h_t(x)- \min_{x,t} h_t(x) $ and denote by $\kappa_t^h$ the smooth function determined by $(\varphi_t^h)^\ast\alpha=\kappa_t^h\cdot \alpha$, where $\alpha$ is the contact form. If $h,g$ are admissible contact Hamiltonians such that 
\[g-h\geqslant \max_{a\in[0,1]} \op{osc}(f^a\cdot \kappa^h),\]
then $g\# f^a\geqslant h\#f^a$ for all $a\in[0,1]$. Consequently, the following diagram commutes
\[\begin{tikzcd}
    HF_\ast(h) \arrow{r}{\mathcal{B}(\varphi^f,\sigma)}\arrow{d}& HF_\ast(h\#f)\arrow{d}\\
    HF_\ast(g) \arrow{r}{\mathcal{B}(\varphi^f,\sigma)}& HF_\ast(g\#f).
\end{tikzcd}\]
In the diagram, the vertical arrows represent the continuation maps. Hence, there exists an isomorphism $\mathcal{B}(\varphi^f,\sigma): SH_\ast(W)\to SH_\ast(W)$ such that the diagram
\[\begin{tikzcd}
    SH_\ast(W) \arrow{r}{\mathcal{B}(\varphi^f,\sigma)}& SH_\ast(W)\\
    HF_\ast(h) \arrow{u}\arrow{r}{\mathcal{B}(\varphi^f,\sigma)}& HF_\ast(h\#f),\arrow{u}
\end{tikzcd}\]
whose vertical arrows are canonical morphisms, commutes for every admissible contact Hamiltonian $h$.

\printbibliography
\end{document}